\documentclass[a4paper,11pt]{amsart}

\usepackage[T1]{fontenc}
\usepackage{amssymb,amsfonts,amsmath,amsthm}
\usepackage{url}
\usepackage{color}
\usepackage{setspace}
\usepackage{enumerate}
\usepackage{epsfig,graphicx,psfrag}
\usepackage{asymptote,tikz}
\usepackage{pinlabel}

\input{xy}
\xyoption{all}
\objectmargin={3mm}

\newcounter{notes}
\newtheorem{theorem}{Theorem}
\newtheorem{proposition}[theorem]{Proposition}
\newtheorem{corollary}[theorem]{Corollary}
\newtheorem{lemma}[theorem]{Lemma}

\newtheorem{observation}[theorem]{Observation}

\theoremstyle{definition}
\newtheorem{definition}[theorem]{Definition}
\newtheorem{remark}[theorem]{Remark}

\newtheoremstyle{theoremwithref}{}{}{\itshape}{}{\bfseries}{.}{.5em}{#1 #2 #3}
\theoremstyle{theoremwithref}

\newcommand{\ie}{i.e.\ }
\newcommand{\eg}{e.g.\ }
\newcommand{\resp}{\text{resp.\ }}
\newcommand{\C}{\mathbb{C}}
\newcommand{\R}{\mathbb{R}}
\newcommand{\Z}{\mathbb{Z}}
\newcommand{\N}{\mathbb{N}}
\newcommand{\SL}{\mathrm{SL}(2,\mathbb{R})}
\newcommand{\PSL}{\mathrm{PSL}(2,\mathbb{R})}

\newcommand{\PGL}{\mathrm{PGL}(2,\mathbb{R})}

\newcommand{\PO}{\mathrm{PO}}
\newcommand{\Adm}{\mathrm{Adm}_g}

\newcommand{\Repfd}{\mathrm{Rep}_g^{\mathrm{fd}}}
\newcommand{\Repnfd}{\mathrm{Rep}_g^{\mathrm{nfd}}}
\newcommand{\Hom}{\mathrm{Hom}}
\newcommand{\HH}{\mathbb{H}}
\newcommand{\Isom}{\mathrm{Isom}}
\newcommand{\AdS}{\mathrm{AdS}}
\newcommand{\Lip}{\mathrm{Lip}}

\title[Compact AdS $3$-manifolds and folded hyperbolic structures]{Compact anti-de Sitter $3$-manifolds and folded hyperbolic structures on surfaces}

\author{Fran\c{c}ois Gu\'eritaud}
\address{CNRS and Universit\'e Lille 1, Laboratoire Paul Painlev\'e, 59655 Villeneuve d'Ascq Cedex, France}
\email{francois.gueritaud@math.univ-lille1.fr}

\author{Fanny Kassel}
\address{CNRS and Universit\'e Lille 1, Laboratoire Paul Painlev\'e, 59655 Villeneuve d'Ascq Cedex, France}
\email{fanny.kassel@math.univ-lille1.fr}

\author{Maxime Wolff}
\address{Universit\'e Paris 6, Institut de Math\'ematiques de Jussieu, 4 place Jussieu, 75252
Paris Cedex 05, France}
\email{wolff@math.jussieu.fr}

\thanks{The authors are partially supported by the Agence Nationale de la Recherche under the grants DiscGroup (ANR-11-BS01-013), ETTT (ANR-09-BLAN-0116-01), ModGroup (ANR-11-BS01-0020), SGT (ANR-11-BS01-0018), and through the Labex CEMPI (ANR-11-LABX-0007-01)}

\begin{document}

\numberwithin{theorem}{section}
\numberwithin{equation}{section}

\begin{abstract}
We prove that any nonabelian, non-Fuchsian representation of a surface group into $\PSL$ is the holonomy of a folded hyperbolic structure on the surface.
Using similar ideas, we establish that any non-Fuchsian representation $\rho$ of a surface group into $\PSL$~is~strictly dominated by some Fuchsian representation~$j$, in the sense that the hyperbolic translation lengths for~$j$ are uniformly larger than for~$\rho$; conversely, any Fuchsian representation~$j$ strictly dominates some non-Fuchsian representation~$\rho$, whose Euler class can be prescribed.
This has applications to compact anti-de Sitter $3$-manifolds.
\end{abstract}

\maketitle

\section{Introduction}

Let $\Sigma_g$ be a closed, connected, oriented surface of genus~$g$, with fundamental group $\Gamma_g=\pi_1(\Sigma_g)$, and let $\Repfd$ (\resp $\Repnfd$) be the set of conjugacy classes of Fuchsian (resp.\ non-Fuchsian) representations of $\Gamma_g$ into $\PSL$.
The letters ``fd'' stand for ``faithful, discrete''.
By work of Goldman \cite{gol88}, the space $\Hom(\Gamma_g,\PSL)$ of representations of $\Gamma_g$ into $\PSL$ has $4g-3$ connected components, indexed by the values of the Euler class
$$\mathrm{eu} :\ \Hom(\Gamma_g,\PSL) \longrightarrow \{2-2g,\ldots, -1,0,1, \dots 2g-2\} .$$
In the quotient, $\Repfd$ consists of the two connected components of\linebreak $\Hom(\Gamma_g,\PSL)/\PSL$ of extremal Euler class, and $\Repnfd$ of all the other components.

\subsection{Strictly dominating representations}

For any $g\in\PSL$, let
\begin{equation} \label{eqn:deflambda}
\lambda(g) := \inf_{x\in\HH^2} d(x,g\cdot x) \geq 0
\end{equation}
be the translation length of $g$ in the hyperbolic plane~$\HH^2$.
The function $\lambda : \PSL\rightarrow\R^+$ is invariant under conjugation.
We say that an element $[j]\in\Repfd$ \emph{strictly dominates} an element $[\rho]\in\Repnfd$ if
\begin{equation}\label{eqn:propcrit0}
\sup_{\gamma\in\Gamma_g\smallsetminus\{ 1\} }\ \frac{\lambda(\rho(\gamma))}{\lambda(j(\gamma))} < 1 .
\end{equation}
Note that \eqref{eqn:propcrit0} can never hold when $j$ and~$\rho$ are both Fuchsian \cite{thu86}.
In this paper we prove the following.

\begin{theorem}\label{thm:domin}
Any $[\rho]\in\Repfd$ is strictly dominated by some $[j]\in\Repfd$.
Any $[j]\in\Repfd$ strictly dominates some $[\rho]\in\Repnfd$, whose Euler class can be prescribed.
\end{theorem}

The first statement of Theorem~\ref{thm:domin} has been simultaneously and independently obtained by Deroin--Tholozan \cite{dt13}, using more analytical methods.
Their paper deals, more generally, with representations of $\Gamma_g$ into the isometry group of any complete, simply connected Riemannian manifold with sectional curvature $\leq -1$.
They also announce a version for general $\mathrm{CAT}(-1)$ spaces.
The present methods, relying as they do on the Toponogov theorem (see Lemma \ref{lem:stretchlocus} below), could likely extend to this general setting as well.

Our approach is constructive, using folded (or pleated) hyperbolic surfaces, as we now explain.

\subsection{Folded hyperbolic surfaces}

Pleated hyperbolic surfaces were introduced by Thurston \cite{thu80} and play an important role in the theory of hyperbolic $3$-manifolds.
A \emph{folded hyperbolic surface} is a pleated surface with all angles equal to $0$ or~$\pi$, whose holonomy takes values in $\PSL$ (see Section~\ref{subsec:reminders-fold}).
It is easy to see \cite[Prop.\,2.1]{thu86} that the holonomy of a (nontrivially) folded hyperbolic structure on~$\Sigma_g$ belongs to $\Repnfd$.
In order to establish Theorem~\ref{thm:domin}, we prove that the converse holds for nonabelian representations.

\begin{theorem}\label{thm:nfdfolded}
An element of $\Repnfd$ has a nonabelian image if and only if it is the holonomy of a folded hyperbolic structure on~$\Sigma_g$.
\end{theorem}

This result seems to have been known to experts since the work of Thurston \cite{thu80}, but to our knowledge it is not stated nor proved in the literature.

We construct the folded hyperbolic structures of Theorem~\ref{thm:nfdfolded} explicitly, folding along geodesic laminations that are the union of simple closed curves and of maximal laminations of some pairs of pants (Proposition~\ref{prop:signs}).
More precisely, given a nonabelian, non-Fuchsian representation~$\rho$, we use a result of Gallo-Kapovich-Marden \cite{gkm00} to find a pants decomposition of~$\Sigma_g$ such that the restriction of $\rho$ to any pair of pants $P$ is nonabelian and maps any cuff to a hyperbolic element.
Folding along a certain maximal lamination in~$P$ then gives a simple dictionary between the representations of the fundamental group of~$P$ that have Euler class $0$ and those that have Euler class $\pm 1$ (Lemma~\ref{lem:pantsfold}).
The converse direction in Theorem~\ref{thm:nfdfolded} is elementary (Observation~\ref{obs:pleated-nonab}).

\subsection{Idea of the proof of Theorem~\ref{thm:domin}}

If $[\rho]\in\Repnfd$ is the holonomy of a folded hyperbolic structure on~$\Sigma_g$, then the holonomy $[j_0]\in\Repfd$ of the corresponding unfolded hyperbolic structure clearly dominates~$[\rho]$ in the sense that $\lambda(\rho(\gamma))\leq\lambda(j_0(\gamma))$ for all $\gamma\in\Gamma_g$; in fact,
$$\sup_{\gamma\in\Gamma_g\smallsetminus\{ 1\} }\ \frac{\lambda(\rho(\gamma))}{\lambda(j_0(\gamma))} = 1$$
since any minimal component of the folding lamination can be approximated by simple closed curves.
In order to prove Theorem~\ref{thm:domin} we need to make the domination \emph{strict}.

To establish the first statement, the idea is, for any $[\rho]\in\Repnfd$, to consider the holonomy $[j_0]\in\Repfd$ of the unfolded hyperbolic structure given by Theorem~\ref{thm:nfdfolded}, and to lengthen the closed curves (close to being) contained in the folding lamination while simultaneously not shortening too much the other curves.
To do this, we work independently in each ``folded subsurface'' of~$\Sigma_g$, which is a compact surface with boundary endowed with a hyperbolic structure induced by~$j_0$: in each such subsurface we use a \emph{strip deformation} construction due to Thurston \cite{thu86}, which consists in adding hyperbolic strips to obtain a new hyperbolic metric with longer boundary components.
We then glue back along the boundaries, after making sure that the lengths agree.

The second statement is easier in that it does not rely on Theorem~\ref{thm:nfdfolded}.
Starting with an element $[j]\in\Repfd$, we choose a pants decomposition of~$\Sigma_g$ along which to fold.
To make sure that the cuffs of the pairs of pants will get contracted, we first deform $j$ slightly by \emph{negative strip deformations} into another element $[j_0]\in\Repfd$ with shorter cuffs, in such a way that the other curves do not get much longer.
Folding $j_0$ then gives an element $[\rho]\in\Repnfd$ which is strictly dominated by~$[j]$.

\subsection{An application to compact anti-de Sitter $3$-manifolds}

Theorem~\ref{thm:domin} has consequences on the theory of compact \emph{anti-de Sitter} $3$-manifolds.
These are the compact Lorentzian $3$-manifolds of constant negative curvature, \ie the Lorentzian analogues of the compact hyperbolic $3$-manifolds.
They are locally modeled on the $3$-dimensional anti-de Sitter space
$$\AdS^3 = \PO(2,2)/\PO(2,1),$$
which identifies with $\PSL$ endowed with the natural Lorentzian structure induced by the Killing form of its Lie algebra.
The identity component of the isometry group of $\AdS^3$ is $\PSL\times\PSL$, acting on $\PSL\simeq\AdS^3$ by left and right multiplication: $(g_1,g_2) \cdot g = g_1 g g_2^{-1}$.

By \cite{kli96}, all compact anti-de Sitter $3$-manifolds are geodesically complete.
By \cite{kr85} and the Selberg lemma \cite[Lem.\,8]{sel60}, they are quotients of $\PSL$ by torsion-free discrete subgroups $\Gamma'$ of $\PSL\times\PSL$ acting properly discontinuously, up to a finite covering; moreover, the groups $\Gamma'$ are graphs of the form
$$\Gamma_g^{j,\rho} = \{ (j(\gamma),\rho(\gamma))~|~\gamma\in\Gamma_g\} ,$$
for some $g\geq 2$, where $j,\rho\in\Hom(\Gamma_g,\PSL)$ are representations with $j$ Fuchsian, up to switching the two factors of $\PSL\times\PSL$.
In particular, $\Gamma'\backslash\AdS^3$ is Seifert fibered over a hyperbolic base (see \cite[\S\,3.4.2]{salPhD}).

Following \cite{sal00}, we shall say that a pair $(j,\rho)\in\Hom(\Gamma_g,\PSL)^2$ with $j$ Fuchsian is \emph{admissible} if the action of $\Gamma_g^{j,\rho}$ on $\AdS^3$ is properly discontinuous.
Clearly, $(j,\rho)$ is admissible if and only if its conjugates under $\PSL\times\PSL$ are.
Therefore, in order to understand the moduli space of compact anti-de Sitter $3$-manifolds, we need to understand, for any $g\geq 2$, the space
$$\Adm \subset \Repfd \times \Hom(\Gamma_g,\PSL)/\PSL$$
of conjugacy classes of admissible pairs $(j,\rho)$.

Examples of admissible pairs are readily obtained by taking $\rho$ to be constant, or more generally with bounded image; the corresponding quotients of $\AdS^3$ are called \emph{standard}.
The first nonstandard examples were constructed by Goldman \cite{gol85} by deformation of standard ones --- a technique later generalized by Kobayashi \cite{kob98}.
Salein \cite{sal00} constructed the first examples of admissible pairs $(j,\rho)$ with $\mathrm{eu}(\rho)\neq 0$; he actually constructed examples where $\mathrm{eu}(\rho)$ can take any nonextremal value.
A necessary and sufficient condition for admissibility was given in \cite{kasPhD}: 
a pair $(j,\rho)$ with $j$ Fuchsian is admissible if and only if $\rho$ is strictly dominated by~$j$ in the sense of \eqref{eqn:propcrit0}.
In particular, by \cite{thu86},
$$\Adm \subset \Repfd \times \Repnfd .$$
This properness criterion was extended in \cite{gk13} to quotients of $\PO(n,1)=\mathrm{Isom}(\HH^n)$ by discrete subgroups of $\PO(n,1)\times\PO(n,1)$ acting~by~left and right multiplication, for arbitrary $n\geq 2$ (recall that $\PSL\simeq\nolinebreak\PO(2,1)_0$), and in \cite{ggkw} to quotients of any simple Lie group $G$ of real rank~$1$.

By completeness \cite{kli96} of compact anti-de Sitter manifolds, the Ehresmann--Thur\-ston principle (see \cite{thu80}) implies that $\Adm$ is open in $\Repfd\times\Repnfd$.
Moreover, $\Adm$ has at least $4g-5$ connected components, as Salein's examples show.
Using the fact that the two connected components of $\Repfd$ are conjugate under $\PGL$, we can reformulate Theorem~\ref{thm:domin} as follows.

\begin{corollary}\label{cor:main}
The projections of $\Adm$ to $\Repfd$ and to $\Repnfd$ are both surjective.
Moreover, for any connected components $\mathcal{C}_1$ of $\Repfd$ and $\mathcal{C}_2$ of $\Repnfd$, the projections of $\Adm\cap (\mathcal{C}_1\times\mathcal{C}_2)$ to $\mathcal{C}_1$ and to $\mathcal{C}_2$ are both surjective.
\end{corollary}

The topology of $\Adm$ is still unknown, but we believe that Corollary~\ref{cor:main} (and the ideas behind its proof) could be used to prove that $\Adm$ is homeomorphic to $\Repfd\times\Repnfd$.
Using the work of Hitchin \cite[Th.\,10.8 \& Eq.\,10.6]{hit87}, this would give the homeomorphism type of the connected components of $\Adm$ corresponding to $\mathrm{eu}(\rho)\neq 0$.

Furthermore, it would be interesting to obtain a geometric and combinatorial description of the fibers of the second projection $\Adm\rightarrow\Repnfd$.
Such a description is given in \cite{dgk13}, in terms of the arc complex, in the different case that $j$ and~$\rho$ are the holonomies of two convex cocompact hyperbolic structures on a given \emph{noncompact} surface.

\subsection{Organization of the paper}

In Section~\ref{sec:reminders} we recall some facts about Lipschitz maps, folded hyperbolic structures, and the Euler class.
Section~\ref{sec:folded} is devoted to the proof of Theorem~\ref{thm:nfdfolded}, and Section~\ref{sec:surj} to that of Theorem~\ref{thm:domin}.

\subsection*{Acknowledgements}

The third author would like to thank Antonin Guilloux, Julien March\'e, and Richard Wentworth for their support and for helpful conversations.

\section{Reminders and useful facts}\label{sec:reminders}

\subsection{Lipschitz maps and their stretch locus}\label{subsec:rem-Lip}

In the whole paper, we denote by $d$ the metric on the real hyperbolic plane~$\HH^2$.
For a Lipschitz map $f :\nolinebreak\HH^2\rightarrow\nolinebreak\HH^2$ and a point $p\in\HH^2$, we set
\begin{itemize}
  \item $\Lip(f) := \sup_{q\neq q'} d(f(q),f(q'))/d(q,q') \geq 0$ (Lipschitz constant);
  \item $\Lip_p(f) := \inf_{\mathcal{U}} \Lip(f|_{\mathcal{U}}) \geq 0$, where $\mathcal{U}$ ranges over all neighborhoods of $p$ in~$\HH^2$ (local Lipschitz constant).
\end{itemize}
The function $p\mapsto\Lip_p(f)$ is upper semicontinuous:
$$\Lip_p(f) \geq \underset{\scriptscriptstyle n\rightarrow +\infty}{\overline{\lim}} \Lip_{p_n}(f)$$
for any sequence $(p_n)_{n\in\N}$ converging to~$p$.
The following is straightforward.

\begin{remark}\label{rem:local-Lip}
For any rectifiable path $\mathcal{L}\subset\HH^2$,
$$\mathrm{length}(f(\mathcal{L})) \leq \sup_{p\in\mathcal{L}}\ \Lip_p(f) \cdot \mathrm{length}(\mathcal{L}) .$$
In particular, if $\Lip_p(f)\leq C$ for all $p$ in a convex set~$K$, then $\Lip(f|_K)\leq C$.
\end{remark}

The following result is contained in \cite[Th.\,5.1]{gk13}; it relies on the Toponogov theorem, a comparison theorem relating the curvature to the divergence rate of geodesics (see \cite[Lem.\,II.1.13]{bh99}).

\begin{lemma}\label{lem:stretchlocus}
Let $\Gamma$ be a torsion-free, finitely generated, discrete group and $(j,\rho)\in\Hom(\Gamma,\PSL)^2$ a pair of representations with $j$ convex cocompact.
Suppose the infimum of Lipschitz constants for all $(j,\rho)$-equivariant maps $f : \HH^2\rightarrow\HH^2$ is~$1$, and the space $\mathcal{F}$ of maps achieving this infimum is nonempty.
Then there exists a nonempty, $j(\Gamma)$-invariant geodesic lamination $\widetilde{\Lambda}$ of~$\HH^2$ such that
\begin{itemize}
  \item any leaf of~$\widetilde{\Lambda}$ is isometrically preserved by all maps $f\in\mathcal{F}$;
  \item any connected component of $\HH^2\smallsetminus\widetilde{\Lambda}$ is either isometrically preserved by all $f\in\mathcal{F}$, or consists entirely of points $p$ at which $\Lip_p(f)<1$ for some $f\in\mathcal{F}$ (independent of~$p$).
\end{itemize}
\end{lemma}

\begin{definition}\label{def:stretch-locus}
The union of~$\widetilde{\Lambda}$ and of the connected components of $\HH^2\smallsetminus\widetilde{\Lambda}$ that are isometrically preserved by all $f\in\mathcal{F}$ is called the \emph{stretch locus} of $(j,\rho)$.
\end{definition}

By \emph{$(j,\rho)$-equivariant} we mean $f(j(\gamma)\cdot p)=\rho(\gamma)\cdot p$ for all $\gamma\in\Gamma$~and $p\in\nolinebreak\HH^2$.
The space $\mathcal{F}$ is always nonempty if $\rho$ is nonelementary \cite[Lem.\,4.11]{gk13}.
If $j$ and~$\rho$ are conjugate under $\PGL$, then the stretch locus of $(j,\rho)$ is the preimage of the convex core of $j(\Gamma)\backslash\HH^2$.

A technical tool for understanding the stretch locus is a procedure for averaging Lipschitz maps (see \cite[\S\,2.5]{gk13}), under which $\Lip_p$ behaves as it would for the barycenter of maps between affine Euclidean spaces.
In Section~\ref{subsec:uniform-Lip}, we shall use this procedure with a partition of unity, as follows.

Let $\psi_0,\dots,\psi_n : \HH^2\rightarrow [0,1]$ be Lipschitz functions inducing a partition of unity on a subset $X$ of~$\HH^2$, subordinated to an open covering $B_0\cup\nolinebreak\ldots\cup\nolinebreak B_n\supset\nolinebreak X$.
For $0\leq i\leq n$, let $\varphi_i : B_i\rightarrow \HH^2$ be a Lipschitz map.
For $p\in X$, let $I(p)$ be the collection of indices $i$ such that $p\in B_i$.
Let $\sum_{i=0}^n \psi_i\,\varphi_i : X\rightarrow\HH^2$ be the map sending any $p\in X$ to the minimizer in~$\HH^2$ of
$$\sum_{i\in I(p)} \psi_i(p) \, d(\,\cdot\, , \varphi_i(p))^2 .$$
The following is contained in \cite[Lem.\,2.13]{gk13}.

\begin{lemma}\label{lem:partofunity}
The averaged map $\varphi:=\sum_{i=0}^n \psi_i\,\varphi_i$ satisfies the ``Leibniz rule''
$$\Lip_p(\varphi)\leq \sum_{i\in I(p)}  \big(\Lip_p(\psi_i)\,R(p) + \psi_i(p)\,\Lip_p(\varphi_i)\big)$$
for all $p\in X$, where $R(p)$ is the diameter of the set $\{\varphi_i(p)\,|\,i\in I(p)\} $.
\end{lemma}

Understanding the stretch locus has led to different equivalent conditions for admissibility in \cite{kasPhD,gk13}:

\begin{theorem}\label{thm:contcont}
Let $\Gamma$ be a torsion-free, finitely generated, discrete group.
A pair $(j,\rho)\in\Hom(\Gamma,\PSL)^2$ is admissible if and only if (up to switching the two factors) $j$ is injective and discrete and $\rho$ is ``uniformly contracting'' compared to~$j$, which means that 
\begin{enumerate}[(i)]
  \item there is a $(j,\rho)$-equivariant Lipschitz map $f : \HH^2\rightarrow\HH^2$ with $\Lip(f)<1$, 
  \item or equivalently that $\rho$ is strictly dominated by~$j$:
  $$\sup_{\gamma\in\Gamma\ \mathrm{with}\ \lambda(j(\gamma))>0}\ \frac{\lambda(\rho(\gamma))}{\lambda(j(\gamma))} < 1 ,$$
  \item or equivalently, if $j(\Gamma)$ does not contain any parabolic element, that $\rho$ is strictly dominated by~$j$ in restriction to simple closed curves:
  $$\sup_{\gamma\in\Gamma_s}\ \frac{\lambda(\rho(\gamma))}{\lambda(j(\gamma))} < 1 ,$$
  where $\Gamma_s$ is the set of nontrivial elements of~$\Gamma$ corresponding to simple closed curves on the surface $j(\Gamma)\backslash\HH^2$.
\end{enumerate}
\end{theorem}

The implications $(i)\Rightarrow (ii)\Rightarrow (iii)$ are immediate modulo the following easy remark (see \cite[Lem.\,4.5]{gk13}); the implication $(iii)\Rightarrow (i)$ is nontrivial and relies on Lemma~\ref{lem:stretchlocus}.

\begin{remark}\label{rem:Clambda-CLip}
Let $\Gamma$ be a discrete group and $(j,\rho)\in\Hom(\Gamma,\PSL)^2$ a pair of representations.
For any $\gamma\in\Gamma$ and any $(j,\rho)$-equivariant Lipschitz map $f : \HH^2\rightarrow\HH^2$,
$$\lambda(\rho(\gamma)) \leq \Lip(f) \, \lambda(j(\gamma)) .$$
\end{remark}

\subsection{Folded hyperbolic structures} \label{subsec:reminders-fold}

Let $\Sigma$ be a connected, oriented surface of negative Euler characteristic, possibly with boundary, and let $\Gamma=\pi_1(\Sigma)$ be its fundamental group.
Recall from \cite[\S\,7]{bon96} that a \emph{pleated hyperbolic structure} on~$\Sigma$ is a quadruple $(j,\rho,\Upsilon,f)$ where
\begin{itemize}
  \item $j\in\Hom(\Gamma,\PSL)$ is the holonomy of a hyperbolic structure on~$\Sigma$;
  \item $\rho\in\Hom(\Gamma,\mathrm{PSL}(2,\C))$ is a representation;
  \item $\Upsilon$ is a geodesic lamination on~$\Sigma$;
  \item $f : \HH^2\rightarrow\nolinebreak\HH^3$ is a $(j,\rho)$-equivariant, continuous map whose restriction to any connected component of $\HH^2\smallsetminus\widetilde{\Upsilon}$ is an isometric embedding.
  (Here we denote by $\widetilde{\Upsilon}\subset\HH^2$ the preimage of $\Upsilon\subset\Sigma\simeq j(\Gamma)\backslash\HH^2$.)
\end{itemize}
The representation $\rho$ is called the \emph{holonomy} of the pleated hyperbolic structure.
The closures of connected components of $\HH^2\smallsetminus\widetilde{\Upsilon}$ are called the \emph{plates} (or \emph{tiles}).
Note that $f$ is $1$-Lipschitz.
For any $g,h\in\PGL$,
$$\big(gj(\cdot)g^{-1} , h\rho(\cdot)h^{-1} , \Upsilon , h\circ f\circ g^{-1}\big)$$
is still a pleated hyperbolic structure on~$\Sigma$; we shall say that $h\rho(\cdot)h^{-1}$ is a \emph{folding} of $gj(\cdot)g^{-1}$.
We make the following observation.

\begin{observation}\label{obs:pleated-nonab}
Suppose that $\Sigma$ is compact (without boundary).
For any pleated hyperbolic structure $(j,\rho,\Upsilon,f)$ on~$\Sigma$, the~group $\rho(\Gamma)$ is nonabelian.
\end{observation}

\begin{proof}
Consider a nondegenerate ideal triangle $T$ of~$\HH^2$ entirely contained in one plate.
Let $(p_n)_{n\in\N}$ be a sequence of points of~$T$ going to infinity.
Since $\Sigma$ is compact, there exist $R>0$ and a sequence $(\gamma_n)_{n\geq 1}$ of elements of~$\Gamma$ such that $d(p_n,j(\gamma_n)\cdot p_0)\leq R$ for all $n\in\N$.
Since $f$ is $(j,\rho)$-equivariant and $1$-Lipschitz,
$$d(\rho(\gamma_n)\cdot f(p_0),f(p_n)) \leq d(j(\gamma_n)\cdot p_0,p_n) \leq R$$
for all $n\in\N$.
Applying this to sequences $(p_n)$ converging to the three ideal vertices of~$T$, and using the fact that the restriction of $f$ to~$T$ is an isometry, we see that the limit set of $\rho(\Gamma)$ contains at least three points.
In particular, $\rho(\Gamma)$ is nonabelian.
\end{proof}

We shall also use the following elementary remark.

\begin{remark}\label{rem:spiral-boundary}
Let $(j,\rho,\Upsilon,f)$ be a pleated hyperbolic structure on~$\Sigma$.
If some leaf of~$\Upsilon$ spirals to a boundary component of~$\Sigma$ corresponding to an element $\gamma\in\Gamma$, then $\lambda(j(\gamma))=\lambda(\rho(\gamma))$, where $\lambda : \mathrm{PSL}(2,\C)\rightarrow\R^+$ is the translation length function in~$\HH^3$ extending \eqref{eqn:deflambda}.
\end{remark}

Any pleated hyperbolic structure $(j,\rho,\Upsilon,f)$ on~$\Sigma$ defines a \emph{bending cocycle}, \ie a map $\beta$ from the set of pairs of plates to $\R/2\pi\Z$ which is symmetric and additive:
$$\beta(P,Q)=\beta(Q,P) \quad\mathrm{and}\quad \beta(P,Q)+\beta(Q,R)=\beta(P,R)$$
for all plates $P,Q,R$.
Intuitively, $\beta(P,Q)$ is the total angle of pleating encountered when traveling from $f(P)$ to $f(Q)$ along $f(\HH^2)$ in~$\HH^3$.
Conversely, to any bending cocycle, Bonahon associates a pleated surface (see \cite[\S\,8]{bon96}).

In this paper we consider a special case of pleated surfaces $(j,\rho,\Upsilon,f)$, namely those for which $f$ takes values in a copy of $\HH^2$ inside~$\HH^3$ (\ie a totally geodesic plane) and $\rho$ takes values in $\Isom^+(\HH^2)=\PSL$.
In this case, we speak of a {\em folded hyperbolic structure} on~$\Sigma$.
The map $f$ defines a \emph{coloring} of $\Sigma\smallsetminus\Upsilon$, namely a $j(\Gamma)$-invariant function $\widetilde{c}$ from the set of plates to $\{ -1,1\}$: we set $\widetilde{c}(P)=1$ if the restriction of $f$ to $P$ is orientation-preserving, and $\widetilde{c}(P)=-1$ otherwise.
Note that the bending cocycle of a folded hyperbolic structure is valued in $\{0,\pi\}$: for all plates $P$ and~$Q$,
\begin{equation}\label{eqn:cocycle-coloring}
\beta(P,Q) = \frac{1 - \widetilde{c}(P) \, \widetilde{c}(Q)}{2} \, \pi \ \in \{ 0,\pi\} .
\end{equation}
The coloring $\widetilde{c}$ descends to a continuous, locally constant function $c$ from $\Sigma\smallsetminus\nolinebreak \Upsilon$ to $\{ -1,1\} $.
Conversely, any such function, after lifting it to a coloring $\widetilde{c}$ from the set of connected components of $\HH^2\smallsetminus\widetilde{\Upsilon}$ to $\{ -1,1\} $, defines a bending cocycle on $\HH^2\smallsetminus\widetilde{\Upsilon}$ by the formula \eqref{eqn:cocycle-coloring}; this bending cocycle, in turn, defines a folded hyperbolic structure on~$\Sigma$ by the work of Bonahon \cite{bon96}.

\subsection{The Euler class}\label{subsec:Euler}

We now give a brief introduction to the Euler class, along the lines of \cite[\S\,2.3.3]{wol11}.
For details and complements we refer to \cite{ghy01} or \cite[\S\,2]{cal04}.

As in the introduction, let $\Sigma_g$ be a closed, connected, oriented surface of genus $g\geq 2$, with fundamental group~$\Gamma_g$.
The Euler class of a representation $\rho\in\Hom(\Gamma_g,\PSL)$ measures the obstruction to lifting $\rho$ to the universal cover $\widetilde{\mathrm{PSL}}(2,\R)$ of $\PSL$; its parity measures the obstruction to lifting $\rho$ to $\SL$.
To define it, choose a set-theoretic section $s$ of the covering map $\widetilde{\mathrm{PSL}}(2,\R)\rightarrow\PSL$.
Consider a triangulation of~$\Sigma_g$ with a vertex at the basepoint~$x_0$ defining $\Gamma_g=\pi_1(\Sigma_g)$.
Choose a maximal tree in the $1$-skeleton of the triangulation, and for every oriented edge $\sigma$ in this tree, set $\rho(\sigma):=1\in\PSL$.
Any other oriented edge $\sigma'$ corresponds (by closing up in the unique possible way along the rooted tree) to an element $\gamma\in\Gamma_g$; we set $\rho(\sigma'):=\rho(\gamma)\in\PSL$.
The boundary of any oriented triangle $\tau$ of the triangulation can be written as $\sigma_1\sigma_2\sigma_3$ where $\sigma_1,\sigma_2,\sigma_3$ are oriented edges; we set
$$\mathrm{eu}(\rho)(\tau) := s(\rho(\sigma_1)) \, s(\rho(\sigma_2)) \, s(\rho(\sigma_3)).$$
Summing over triangles $\tau$, this defines an element of $H^2(\Sigma_g,\pi_1(\PSL))$, hence an element of $H^2(\Sigma_g,\Z)$ under the identification $\pi_1(\PSL)\simeq\Z$.
This element $\mathrm{eu}(\rho)\in H^2(\Sigma_g,\Z)$ is called the \emph{Euler class} of~$\rho$.
Its evaluation on the fundamental class in $H_2(\Sigma_g,\Z)$ is an integer, which we still call the Euler class of~$\rho$.
It is invariant under conjugation by $\PSL$, and changes sign under conjugation by $\PGL\smallsetminus\PSL$.

We can also define the Euler class for representations of the fundamental group of a compact, connected, oriented surface $\Sigma$ with boundary, of negative Euler characteristic, when the boundary curves are sent to hyperbolic elements. Indeed, any hyperbolic element $g\in\PSL$ has a canonical lift to $\widetilde{\mathrm{PSL}}(2,\R)$: it belongs to a unique one-parameter subgroup of $\PSL$, which defines a path from the identity to~$g$.
Choose a section $s$ of the projection $\widetilde{\mathrm{PSL}}(2,\R)\rightarrow\PSL$ that maps any hyperbolic element to its canonical lift.
Then the construction above, using triangulations of~$\Sigma$ containing exactly one vertex on each boundary component, defines an Euler class, independent of all choices.

For instance, let $\Sigma$ be an oriented pair of pants with fundamental group $\Gamma=\langle\alpha,\beta,\gamma\,|\,\alpha\beta\gamma=1\rangle$, where $\alpha,\beta,\gamma$ correspond to the three boundary curves, endowed with the orientation induced by the surface.
For any representation $\rho\in\Hom(\Gamma,\PSL)$ with $\rho(\alpha),\rho(\beta),\rho(\gamma)$ hyperbolic,
\begin{equation}\label{eqn:Euler-nb-pants}
\mathrm{eu}(\rho) = s(\rho(\alpha)) \, s(\rho(\beta)) \, s(\rho(\gamma)) \in Z(\widetilde{\mathrm{PSL}}(2,\R)) \, \simeq \, \Z .
\end{equation}
In particular, $\mathrm{eu}(\rho) \in \{-1,0,1\}$, and $|\mathrm{eu}(\rho)|=1$ if and only if $\rho$ is the holonomy of a hyperbolic structure on~$\Sigma$, after possibly reversing the orientation.
If $s'$ is a section of the projection $\SL\rightarrow\PSL$ that maps any hyperbolic element to its lift of positive trace, then \eqref{eqn:Euler-nb-pants} implies
\begin{equation} \label{eqn:eulermod2}
s'(\rho(\alpha)) \, s'(\rho(\beta)) \, s'(\rho(\gamma)) = (-\mathrm{Id})^{\mathrm{eu}(\rho)} .
\end{equation}

By construction, the Euler class is \emph{additive}: if $\Sigma$ is the union of two subsurfaces $\Sigma'$ and $\Sigma''$ glued along curves $\gamma_i$, and if $\rho\in\Hom(\pi_1(\Sigma),\PSL)$ is a representation sending all the curves $\gamma_i$ (and the boundary curves of $\Sigma$, if any) to hyperbolic elements of $\PSL$, then $\mathrm{eu}(\rho)$ is the sum of the Euler classes of the restrictions of $\rho$ to the fundamental groups of $\Sigma'$ and~$\Sigma''$.
This implies that a folded hyperbolic structure defined by a coloring $c$ from the set $\mathcal{P}$ of connected components of $\Sigma\smallsetminus\Upsilon$ to $\{ -1,1\}$ has Euler class $\frac{1}{2\pi}\sum_{P\in\mathcal{P}} c(P)\,\mathcal{A}(P)$ where $\mathcal{A}(P)$ is the area of~$P$.

We shall use the following terminology.

\begin{definition}\label{def:geomrep}
A representation $\rho\in\Hom(\pi_1(\Sigma),\PSL)$ is {\em geometric} if it maps the boundary curves of~$\Sigma$ to hyperbolic elements of $\PSL$ and has extremal Euler class or, equivalently, if it is the holonomy of a hyperbolic structure on~$\Sigma$, after possibly reversing the orientation.
\end{definition}

\subsection{Laminations in a pair of pants}\label{subsec:triskelion}

A hyperbolic pair of pants $\Sigma$ carries only finitely many geodesic laminations, because only 21 geodesics are simple: $3$ closed geodesics (the boundary components), $6$ geodesics spiraling from a boundary component to itself, and $12$ geodesics spiraling from a boundary component to another.
It admits 32 ideal triangulations: 24 of them contain a geodesic spiraling from a boundary component to itself, and the other 8 do not (see Figure~\ref{fig:laminations}).
We shall call the laminations corresponding to these 8 triangulations the \emph{triskelion} laminations of~$\Sigma$; they differ by the spiraling directions of the spikes of the triangles at each boundary component.

\begin{figure}[ht!]
\centering
\includegraphics[width=9cm]{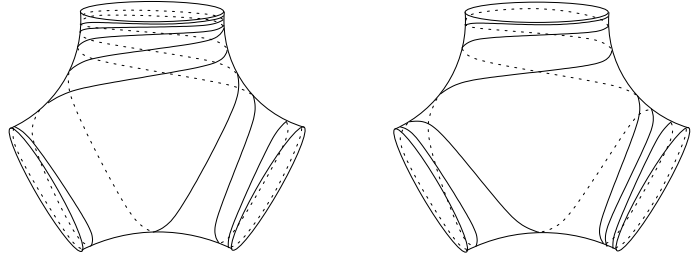}
\caption{A pair of pants carries 24 maximal geodesic laminations containing a geodesic spiraling from a boundary component to itself (left), and 8 triskelion laminations (right).}
\label{fig:laminations}
\end{figure}

\section{Holonomies of folded hyperbolic structures}\label{sec:folded}

Let $\lambda : \PSL\rightarrow\R^+$ be the translation length function \eqref{eqn:deflambda}.
For any representation $\rho\in\Hom(\Gamma_g,\PSL)$, we set
$$\lambda_{\rho} := \lambda\circ\rho :\ \Gamma_g \longrightarrow \R^+ .$$
The function $\lambda_{\rho}$ is identically zero if and only if the group $\rho(\Gamma_g)$ is unipotent or bounded.
The goal of this section is to prove the following.

\begin{proposition}\label{prop:signs}
For any $[\rho]\in\Repnfd$ with $\lambda_{\rho}\not\equiv 0$, there exist elements $[j_0],[j'_0]\in\Repfd$ and a decomposition of~$\Sigma_g$ into pairs of pants, each labeled $-1$, $0$, or~$1$, with the following properties:
\begin{enumerate}
  \item for any representations $j_0,\rho$ in the respective classes $[j_0],[\rho]$, there is a $1$-Lipschitz, $(j_0,\rho)$-equivariant map $f : \HH^2\rightarrow\HH^2$ that is an orientation-preserving (\resp orientation-reversing) isometry in restriction to any connected subset of~$\HH^2$ projecting to a union of pants labeled $1$ (\resp $-1$) in $j_0(\Gamma_g)\backslash\HH^2\simeq\Sigma_g$, and that satisfies $\Lip_p(f)<1$ for any $p\in\HH^2$ projecting to the interior of a pair of pants labeled~$0$;
  \item for any representations $j'_0,\rho$ in the respective classes $[j'_0],[\rho]$, if $\rho$ is nonabelian, then it is a folding of~$j'_0$ along a lamination $\Upsilon$ of~$\Sigma_g$ consisting of all the cuffs together with a triskelion lamination inside each pair of pants labeled~$0$; the coloring $c : \Sigma_g\smallsetminus\Upsilon\rightarrow\{ -1,1\}$ takes the value $1$ (\resp $-1$) on each pair of pants labeled $1$ (\resp $-1$), and both values on each pair of pants labeled~$0$;
  \item $[j_0]$ and~$[j'_0]$ only differ by earthquakes along the cuffs of the pairs of pants of the decomposition.
\end{enumerate}
\end{proposition}

Property~(1) is used to prove Theorem~\ref{thm:domin} in Section~\ref{sec:surj}, while property~(2) is a more precise statement of Theorem~\ref{thm:nfdfolded}.
We refer to Section~\ref{subsec:rem-Lip} for the notation $\Lip_p(f)$ and to Section~\ref{subsec:triskelion} for triskelion laminations.
By additivity (see Section~\ref{subsec:Euler}), the Euler class of~$\rho$ is the sum of the labels of the pairs of pants.

Proposition~\ref{prop:signs} is proved by choosing an appropriate pants decomposition (Section~\ref{subsec:constr}) and understanding the representations of the fundamental group of a pair of pants (Section~\ref{subsec:pants}); these ingredients are brought together in Section~\ref{subsec:provesigns}.
In Section~\ref{subsec:uniform-Lip} we present a variation on Proposition~\ref{prop:signs}.(1), which is later used to prove the second statement of Theorem~\ref{thm:domin}.

\subsection{Pants decompositions}\label{subsec:constr}

Our first ingredient is the following.

\begin{lemma}\label{lem:findloxo}
For any $[\rho]\in\Repnfd$ with $\lambda_{\rho}\not\equiv 0$, there is a pants decomposition of~$\Sigma_g$ such that $\rho$ maps any cuff to a hyperbolic element.
If $\rho$ is~non\-abelian, then we may assume that the restriction of $\rho$ to the fundamental group of any pair of pants of the decomposition is nonabelian.
\end{lemma}

In the case that $\rho$ is nonelementary, Lemma~\ref{lem:findloxo} is contained in the following result of Gallo--Kapovich--Marden \cite[part~A]{gkm00}.

\begin{lemma} \label{lem:gkm}
For any nonelementary $[\rho]\in\Repnfd$, there is a pants decomposition of~$\Sigma_g$ such that the fundamental group of any pair of pants maps injectively to a $2$-generator Schottky group under~$\rho$.
\end{lemma}

To treat the case that $\rho$ is elementary, we use the following terminology.

\begin{definition}\label{def:abelianization}
Let $\rho\in\Hom(\Gamma_g,\PSL)$ be an elementary representation with a fixed point $\xi$ in the boundary at infinity $\partial_{\infty}\HH^2$ of~$\HH^2$.
Choose a geodesic line $\ell$ of~$\HH^2$ with endpoint~$\xi$.
For any $\gamma\in\Gamma$ we may write in a unique way $\rho(\gamma)=au$ where $a\in\PSL$ fixes~$\xi$ and preserves~$\ell$ and $u\in\PSL$ is unipotent or trivial; setting $\rho^{\mathrm{ab}}(\gamma):=a$ defines an abelian representation $\rho^{\mathrm{ab}}\in\Hom(\Gamma_g,\PSL)$.
We call it the \emph{abelianization} of~$\rho$.
Its conjugacy class under $\PSL$ does not depend on the choice of~$\ell$.
\end{definition}

\begin{proof}[Proof of Lemma~\ref{lem:findloxo} when $\rho$ is elementary]
Since $\lambda_{\rho}\not\equiv 0$, the group $\rho(\Gamma_g)$~has a fixed point $\xi\in\partial_{\infty}\HH^2$, and the abelianization of~$\rho$ can be seen as a nonzero element $\omega$ of $H^1(\Sigma_g,\R)$ after identifying the stabilizer of $\xi$ and~$\ell$ with~$(\R,+)$.

To prove the first statement of Lemma~\ref{lem:findloxo}, we must show that there exists a pants decomposition $\Pi$ of $\Sigma_g$ such that $\omega$ is nonzero on every cuff.
First, $\omega$ is nonzero on the homology class of some oriented, \emph{simple} closed curve~$c$: indeed, the generators of the standard presentation
$$\Gamma_g = \Big\langle \alpha_1, \dots, \alpha_g, \beta_1, \dots, \beta_g~\Big|~ \prod_{i=1}^g \, [\alpha_i, \beta_i] = 1\Big\rangle$$
all define simple closed curves and $\omega$ cannot vanish on all the corresponding homology classes.
Note that $c$ is necessarily nonseparating since separating curves are homologically trivial.
It is sufficient to find a pants decomposition $\Pi'$ of~$\Sigma_g$ and a nonseparating simple curve~$c'$ that intersects transversally all the cuffs of~$\Pi'$: then, applying a homeomorphism of~$\Sigma_g$ will enable us to assume that $c'=c$ (see \eg \cite[\S\,1.3.1]{fm12}), and applying enough Dehn twists along $c$ to~$\Pi'$ will yield a pants decomposition $\Pi$ none of whose cuffs annihilate~$\omega$.
Here is how to produce $\Pi'$ and~$c'$.
In genus $g=2$, consider an abstract pair of pants $P$ with boundary components $X, Y, Z$, and disjoint embedded arcs $u,v$ connecting $Y$ to~$X$ and to~$Z$ respectively.
Double $P$ across $X\cup Y\cup Z$ to obtain a copy of~$\Sigma_2$, then perform a \emph{half} Dehn twist around~$Y$: the arcs $u,v$ and their doubles arrange together to produce a single, nonseparating, simple closed curve~$c'$ that intersects transversally each of $X,Y,Z$.
We may take $\Pi'$ to consist of $P$ and its complement.
In arbitrary genus $g\geq 2$, we take a $(g-1)$-fold cyclic cover of this construction with respect to $Z$ (or~$X$). 

We now suppose that $\rho$ is nonabelian and establish the second statement of Lemma~\ref{lem:findloxo} by modifying the pants decomposition $\Pi$ we have just constructed.
Note that no pair of pants of~$\Pi$ is connected to itself.
Therefore, it is enough to prove that if a nondiagonalizable pair of pants $P$ is adjacent to a diagonalizable one $P'$ along a cuff~$U$, then we may modify $P$ and~$P'$ into two nondiagonalizable pairs of pants by applying enough Dehn twists along some simple closed curve of the $4$-punctured sphere $\Sigma:=P\cup U\cup P'$.
(We say that a pair of pants is \emph{diagonalizable} if $\rho$ takes its fundamental group into a conjugate of the group of diagonal matrices of $\PSL$.)
Note that $U$ and its pushforward by the Dehn twists will have the same, \emph{nonzero} image under~$\omega$ (equal to the sum of the images of the two boundary components of $P$ other than~$U$).
The fundamental group of the $4$-punctured sphere $\Sigma$ may be written
$$\pi_1(\Sigma) = \langle \alpha, \beta, \gamma, \delta ~|~ \alpha \beta \gamma \delta = 1 \rangle ,$$
where $\alpha, \beta, \gamma, \delta$ represent the boundary loops of~$\Sigma$, where $w:=\alpha \beta=( \gamma \delta)^{-1}$ represents~$U$, where $\rho(\langle \alpha, \beta \rangle)$ is nondiagonalizable, and where $\rho(\langle\gamma,\delta\rangle)$ is diagonalizable (see Figure~\ref{fig:4holedsph}).
For any $\gamma'\in\Gamma_g$ with $\rho(\gamma')$ hyperbolic, we denote by $\xi_{\rho(\gamma')}\in\partial_{\infty}\HH^2$ the fixed point of~$\rho(\gamma')$ other than the common fixed point~$\xi$.
Then $\xi_{\rho(\alpha)},\xi_{\rho(\beta)},\xi_{\rho(w)}$ are all distinct, while $\xi_{\rho(\gamma)}=\xi_{\rho(\delta)}=\xi_{\rho(w)}$.
Performing $n$ Dehn twists along the curve $V$ represented by $\beta\gamma=(\delta\alpha)^{-1}$ amounts, under suitable basepoint choice, to replacing $\beta$ with $\beta_n:=(\beta\gamma)^n \beta (\beta\gamma)^{-n}$, and $\gamma$ with $\gamma_n:=(\beta\gamma)^n \gamma (\beta\gamma)^{-n}$, while leaving $\alpha$ and $\delta$ unchanged.
The elements $\rho(\beta_n),\rho(\gamma_n)\in\PSL$ are still hyperbolic, and
$$\xi_{\rho(\beta_n)} = \rho(\beta\gamma)^n \cdot \xi_{\rho(\beta)} \quad\quad\mathrm{and}\quad\quad \xi_{\rho(\gamma_n)} = \rho(\beta\gamma)^n \cdot \xi_{\rho(\gamma )}.$$
Note that $\rho(\beta\gamma)\neq 1$ since $\xi_{\rho(\gamma )}=\xi_{\rho(w)}\neq \xi_{\rho(\beta)}$.
Therefore $\rho(\beta\gamma)$ is either a parabolic element, or a hyperbolic element with $\xi_{\rho(\beta\gamma)}\neq \xi_{\rho(\gamma )}$.
It follows that $\xi_{\rho(\gamma_n)}\neq\xi_{\rho(\gamma )}=\xi_{\rho(\delta)}$ for all $n\neq 0$, and that $\xi_{\rho(\beta_n)}\neq \xi_{\rho(\alpha)}$ for all $n\neq 0$ with at most one exception.
Thus, for all $n\neq 0$ with at most one exception, the pairs of pants whose cuffs are represented by $(\alpha,\delta,\beta\gamma)$ and by $(\beta_n,\gamma_n,\beta\gamma)$ are both nondiagonalizable.
This completes the proof of Lemma~\ref{lem:findloxo}.
\end{proof}

\begin{figure}[ht!]
\centering
\labellist
\small\hair 2pt
\pinlabel $\alpha$ [b] at 230 200
\pinlabel $\beta$ [t] at 53 230
\pinlabel $\gamma$ [u] at 55 40
\pinlabel $\delta$ [b] at 243 35
\pinlabel $U$ [t] at 195 138
\pinlabel $V$ [b] at 140 180
\pinlabel $P$ [b] at 0 160
\pinlabel $P'$ [b] at 0 70
\endlabellist
\includegraphics[scale=0.45]{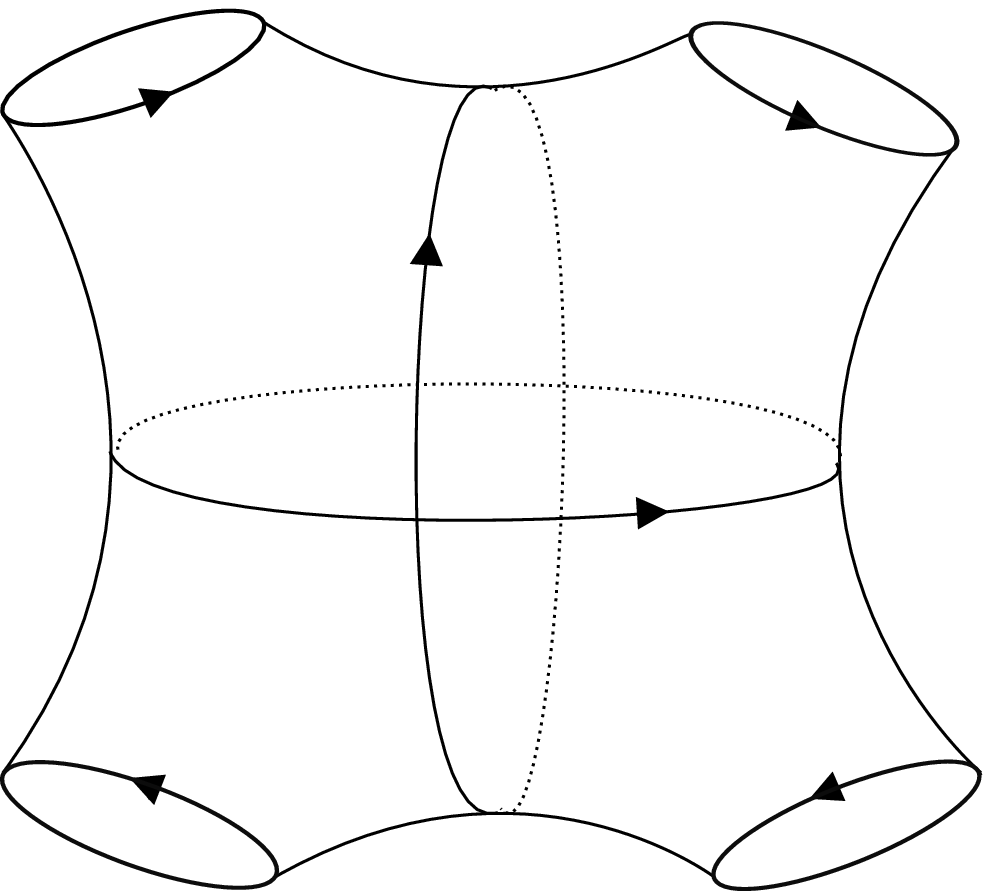}
\caption{The four-holed sphere~$\Sigma$ in the proof of Lemma~\ref{lem:findloxo}}
\label{fig:4holedsph}
\end{figure}

\subsection{Representations of the fundamental group of a pair of pants}\label{subsec:pants}

The following lemma gives a dictionary between the geometric and nongeometric representations (Definition~\ref{def:geomrep}) of the fundamental group of a pair of pants.

\begin{lemma}\label{lem:pantsfold}
Let $\Gamma=\langle \alpha, \beta, \gamma \,|\, \alpha\beta\gamma=1 \rangle$ be the fundamental group of a pair of pants $\Sigma$, with $\alpha, \beta, \gamma$ corresponding to the boundary loops.
\begin{itemize}
  \item For any $a,b,c>0$ such that none is the sum of the other two, there are exactly two representations $\tau\in\Hom(\Gamma,\PSL)$ satisfying
  \begin{equation}\label{eqn:lambda-tau}
  (\lambda_{\tau}(\alpha),\lambda_{\tau}(\beta),\lambda_{\tau}(\gamma)) = (a,b,c) ,
  \end{equation}
  up to conjugation under $\PGL$.
  One of them is geometric (with $|\mathrm{eu}(\tau)|=1$); the other is nongeometric (with $\mathrm{eu}(\tau)=0$), obtained from the geometric one by folding along any of the eight triskelion laminations of~$\Sigma$.
  \item For any $a,b,c>0$ such that one is the sum of the other two, there are exactly four representations $\tau\in\Hom(\Gamma,\PSL)$ satisfying \eqref{eqn:lambda-tau}, up to conjugation under $\PGL$.
  One of them is geometric (with $|\mathrm{eu}(\tau)|=1$).
  The other three are elementary (with $\mathrm{eu}(\tau)=0$): two are nonabelian, the third one is their abelianization.
  Each of the two nonabelian elementary representations is obtained from the geometric one by folding along any of four different triskelion laminations of~$\Sigma$.
\end{itemize}
\end{lemma}

\begin{proof}
Fix $a,b,c>0$.
We first determine the number of conjugacy classes of representations $\tau$ satisfying
\eqref{eqn:lambda-tau}. Set $(A,B,C) := (e^{a/2},e^{b/2},e^{c/2})$, and let $\tau\in\Hom(\Gamma,\PSL)$ satisfy \eqref{eqn:lambda-tau}.
Up to conjugating~$\tau$ by $\PGL$, we can find lifts $\overline{\tau}(\alpha)\in\SL$ of $\tau(\alpha)$ and $\overline{\tau}(\beta)\in\SL$ of~$\tau(\beta)$ of the form
$$\overline{\tau}(\alpha)=\begin{pmatrix} A & 0 \\ 0 & A^{-1} \end{pmatrix}~\text{ and }~\overline{\tau}(\beta)=\begin{pmatrix} B+x & y \\ z & B^{-1}-x \end{pmatrix}$$
with $x,y,z\in\R$.
Since $\alpha$ and~$\beta$ freely generate~$\Gamma$, this determines a lift $\overline{\tau}\in\Hom(\Gamma,\SL)$ of~$\tau$.
The sign $\varepsilon\in \{ \pm 1\}$ of $\mathrm{Tr}(\overline{\tau}(\alpha))\,\mathrm{Tr}(\overline{\tau}(\beta))\,\mathrm{Tr}(\overline{\tau}(\gamma))$ does not depend on the choice of $\overline{\tau}(\alpha),\overline{\tau}(\beta)$.
By \eqref{eqn:Euler-nb-pants}, we have $\mathrm{eu}(\tau)\in\{ -1,0,1\}$, with $|\mathrm{eu}(\tau)|=1$ if and only if $\tau$ is geometric, and by \eqref{eqn:eulermod2}
$$\varepsilon = (-1)^{\mathrm{eu}(\tau)} .$$
The trace of $\overline{\tau}(\gamma)=\overline{\tau}(\alpha\beta)^{-1}$ is
$$A (B + x) + A^{-1} (B^{-1} - x) = \varepsilon (C + C^{-1}) ,$$
hence
$$x = \frac{\varepsilon (C + C^{-1}) - A B - (A B)^{-1}}{A - A^{-1}}$$
is uniquely determined by $A,B,C$ and~$\varepsilon$.
Let $\nu:=(B+x)(B^{-1}-x)$.
Since $\overline{\tau}(\beta)\in\mathrm{SL}(2,\R)$, we have $yz=\nu-1$.
If $\nu\neq 1$, then all pairs $(y,z)$ of reals with product $\nu-1$ can be obtained by conjugating $\overline{\tau}(\alpha), \overline{\tau}(\beta)$ by a diagonal matrix in $\PGL$ (which does not change~$x$): thus $\tau$ is unique up to conjugation once we fix $\varepsilon\in \{-1,1\}$.
If $\nu=1$, then $\overline{\tau}(\beta)$ is either upper or lower triangular, or both, hence three conjugacy classes for~$\tau$, with $\tau(\Gamma)$ consisting respectively of upper triangular, lower triangular, and diagonal matrices.
The condition $\nu=1$ amounts to $(B^{-1}-B-x)x=0$, or equivalently to
$$\left( \frac{B C}{A} - \varepsilon \right) \left( \frac{A C}{B} - \varepsilon \right) \cdot \left( \frac{A B}{C} - \varepsilon \right) (A B C - \varepsilon) = 0 :$$
in other words, $\varepsilon=1$ and one of $a,b,c$ is the sum of the other two.

Let $j\in\Hom(\Gamma,\PSL)$ be geometric (Definition~\ref{def:geomrep}).
For any folding $\rho$ of~$j$ along a triskelion lamination $\Upsilon$ of~$\Sigma$, the functions $\lambda_j$ and $\lambda_{\rho}$ agree on $\{\alpha, \beta, \gamma\}$ (Remark~\ref{rem:spiral-boundary}), and $\rho$ is not conjugate to~$j$ under $\PGL$ because the folding map $f$ is not an isometry (see Section~\ref{subsec:rem-Lip}); therefore, $\mathrm{eu}(\rho)=0$ by the above discussion.
If none of $a,b,c$ is the sum of the other two, then $\rho$ belongs to the unique conjugacy class of representations $\tau$ satisfying \eqref{eqn:lambda-tau} and $\mathrm{eu}(\tau)=0$.
If one of $a,b,c$ is the sum of the other two, then $\rho$ belongs to one of the two conjugacy classes of nonabelian representations $\tau$ satisfying \eqref{eqn:lambda-tau} and $\varepsilon=1$ (Observation~\ref{obs:pleated-nonab}).
The representation $\rho'$ obtained from~$j$ by folding along the image of~$\Upsilon$ under the natural involution of the pair of pants belongs to the other conjugacy class of nonabelian representations $\tau$ satisfying \eqref{eqn:lambda-tau} and $\varepsilon=1$.
The abelianization of $\rho$ or~$\rho'$ is not conjugate to~$j$, hence satisfies \eqref{eqn:lambda-tau} and $\varepsilon=1$ as well.
\end{proof}

\begin{corollary}\label{cor:Lip-pants}
Let $\Gamma=\langle \alpha, \beta, \gamma \,|\, \alpha\beta\gamma=1 \rangle$ be the fundamental group of a pair of pants $\Sigma$, with $\alpha, \beta, \gamma$ corresponding to the boundary loops.
Consider two representations $j,\rho\in\Hom(\Gamma,\PSL)$ with $j$ geometric (Definition~\ref{def:geomrep}), $\rho$ nongeometric, and
$$(\lambda_j(\alpha),\lambda_j(\beta),\lambda_j(\gamma)) = (\lambda_{\rho}(\alpha),\lambda_{\rho}(\beta),\lambda_{\rho}(\gamma)) .$$
Then there exists a $1$-Lipschitz, $(j,\rho)$-equivariant map $f : \HH^2\rightarrow\HH^2$ such~that $\Lip_p(f)<1$ for any $p\in\HH^2$ projecting to a point of $j(\Gamma)\backslash\HH^2$ off the boundary of the convex core.
\end{corollary}

Note that such a map $f$ is necessarily an isometry in restriction to the translation axes of $j(\alpha),j(\beta),j(\gamma)$ in~$\HH^2$.
The convex core of $j(\Gamma)\backslash\HH^2$ naturally identifies with~$\Sigma$.

\begin{proof}
We first assume that $\rho$ is nonabelian.
By Lemma~\ref{lem:pantsfold}, it is obtained from~$j$ by folding along any of at least four of the eight triskelion laminations of~$\Sigma$.
Let $\ell$ be an injectively immersed geodesic that spirals between two boundary components.
If the two boundary components are different, then $\ell$ is contained in only two triskelion laminations, and intersects the others transversely; if the two boundary components are the same, then $\ell$ intersects transversely all triskelion laminations of~$\Sigma$.
In both cases we see that a lift of $\ell$ to~$\HH^2$ cannot be isometrically preserved by all $1$-Lipschitz, $(j,\rho)$-equivariant maps $f : \HH^2\rightarrow\HH^2$ (such maps exist since $\rho$ is a folding of~$j$).
This holds for any~$\ell$, hence shows that the lamination $\widetilde{\Lambda}\subset\HH^2$ of Lemma~\ref{lem:stretchlocus} is contained in (in fact, equal to) the preimage of the boundary of the convex core of $j(\Gamma)\backslash\HH^2$, which identifies with the boundary of~$\Sigma$.
By Lemma~\ref{lem:stretchlocus}, this means that there exists a $1$-Lipschitz, $(j,\rho)$-equivariant map $f : \HH^2\rightarrow\HH^2$ such that $\Lip_p(f)<1$ for all $p\in\HH^2$ projecting to points of $j(\Gamma)\backslash\HH^2$ off the boundary of the convex core.

We now assume that $\rho$ is abelian.
By Lemma~\ref{lem:pantsfold}, it is the abelianization of some representation $\rho'$ that is a folding of~$j$.
The group $\rho'(\Gamma)$ fixes a point $\xi\in\partial_{\infty}\HH^2$, and $\rho(\Gamma)$ preserves a geodesic line $\ell$ of~$\HH^2$ with endpoint~$\xi$.
By postcomposing any $1$-Lipschitz, $(j,\rho')$-equivariant map with the projection onto~$\ell$ along the horospheres centered at~$\xi$, we obtain a $1$-Lipschitz, $(j,\rho)$-equivariant map; moreover, $1$ is the optimal Lipschitz constant by Remark~\ref{rem:Clambda-CLip}.
This shows that the stretch locus (Definition~\ref{def:stretch-locus}) of $(j,\rho)$ is contained in that of $(j,\rho')$, and we conclude as above.
\end{proof}

\begin{remark}
The nonabelian, nongeometric representations in Lemma~\ref{lem:pantsfold} can also be obtained by folding along a nonmaximal geodesic lamination consisting of a unique leaf spiraling from a boundary component to itself.
Folding along a maximal lamination which is not a triskelion gives a representation with values in $\PGL$ and not $\PSL$.
\end{remark}

\subsection{Proof of Proposition~\ref{prop:signs}}\label{subsec:provesigns}

By Lemma~\ref{lem:findloxo}, there is a pants decomposition $\Pi$ of~$\Sigma_g$ such that $\rho$ maps each cuff to a hyperbolic element, and such that if $\rho$ is nonabelian then its restriction to the fundamental group of each pair of pants is nonabelian.
Let $j\in\Hom(\Gamma_g,\PSL)$ be a Fuchsian representation such that $\lambda_j(\gamma)=\lambda_{\rho}(\gamma)$ for all $\gamma\in\Gamma_g$ corresponding to cuffs of pants of~$\Pi$. 
The twist parameters along the cuffs will be adjusted later, so for the moment we choose them arbitrarily.
Let $\mathcal{C}$ be the $j(\Gamma_g)$-invariant (disjoint) union of all geodesics of~$\HH^2$ projecting to the cuffs in $j(\Gamma_g)\backslash\HH^2\simeq\Sigma_g$.
For each pair of pants $P$ in~$\Pi$, choose a subgroup $\Gamma^P$ of~$\Gamma_g$ which is conjugate to $\pi_1(P)$: then $j|_{\Gamma^P}$ is the holonomy of a hyperbolic metric on $P$ with cuff lengths given by~$\lambda_{\rho}$.
Choose a lift $\widetilde{P}\subset\HH^2$ of the convex core of $j(\Gamma_P)\backslash\HH^2$: it is the closure of a connected component of $\HH^2\smallsetminus\mathcal{C}$.
If the restrictions of $j$ and $\rho$ to~$\Gamma^P$ are conjugate by some isometry $f^P$ of~$\HH^2$, then we give $P$ the label $1$ or $-1$, depending on whether $f^P$ preserves the orientation or not.
If the restrictions of $j$ and $\rho$ to~$\Gamma^P$ are not conjugate, then we give $P$ the label~$0$.
In this case,
\begin{itemize}
  \item by Corollary~\ref{cor:Lip-pants}, there is a $1$-Lipschitz, $(j|_{\Gamma^P},\rho|_{\Gamma^P})$-equivariant map $f^P : \widetilde{P}\rightarrow\HH^2$ with $\Lip_p(f^P)<1$ for all $p\notin\partial\widetilde{P}$;
  \item by Lemma~\ref{lem:pantsfold}, if $\rho$ is nonabelian then $\rho|_{\Gamma^P}$ is a folding of~$j|_{\Gamma^P}$ along some triskelion lamination of~$P$; we denote by $F^P : \widetilde{P}\rightarrow\HH^2$ the folding map.
\end{itemize}
Note that in restriction to any connected component of $\partial\widetilde{P}$ (a line), the maps $f^P$ and~$F^P$ are both isometries; they may disagree by a constant shift.

The collection of all maps~$f^P$, extended $(j,\rho)$-equivariantly, piece together to yield a map $f^{\ast} : \HH^2\smallsetminus\mathcal{C}\rightarrow\HH^2$.
The obstruction to extending $f^{\ast}$ by continuity on each geodesic $\ell\subset\mathcal{C}$ is that the maps on either side of~$\ell$ may disagree by a constant shift along~$\ell$ if $\ell$ separates two pairs of pants labeled $(1,-1)$, $(\pm 1,0)$, or $(0,0)$.
This discrepancy $\delta(\ell)\in\R$ is the same on the whole $j(\Gamma_g)$-orbit of~$\ell$.
To correct it, we postcompose $j$ with an earthquake supported on the cuff associated with~$\ell$, of length $-\delta(\ell)$.
We repeat for each $j(\Gamma_g)$-orbit in~$\mathcal{C}$, and eventually obtain a new Fuchsian representation~$j_0$.
By construction, there is a $1$-Lipschitz, $(j_0,\rho)$-equivariant map $f : \HH^2\rightarrow\HH^2$, obtained simply by gluing together isometric translates of the~$f_P$.
This extension $f$ satisfies Proposition~\ref{prop:signs}.(1).

If $\rho$ is nonabelian, then similarly the maps $f^P$ for $P$ labeled $\pm 1$ and $F^P$ for $P$ labeled~$0$ piece together to yield a map $F^{\ast} : \HH^2\smallsetminus\mathcal{C}\rightarrow\HH^2$.
As above, we can modify $j$ by earthquakes into a new Fuchsian representation~$j'_0$, and $F^{\ast}$ by piecewise isometries into a $(j'_0,\rho)$-equivariant, continuous map $F : \HH^2\rightarrow\HH^2$ which is a folding map.
This proves Proposition~\ref{prop:signs}.(2).

Proposition~\ref{prop:signs}.(3) is satisfied by construction.

\subsection{Uniform Lipschitz bounds}\label{subsec:uniform-Lip}

In order to prove the second statement of Theorem~\ref{thm:domin} in Section~\ref{subsec:hallali2}, we shall use the following result, which gives Lipschitz bounds analogous to Proposition~\ref{prop:signs}.(1) but uniform.

\begin{proposition}\label{prop:unif-prop-signs}
For any decomposition $\Pi$ of~$\Sigma_g$ into pairs of pants labeled $-1,0,1$ and any continuous family $(j_t)_{t\geq 0}\subset\Hom(\Gamma_g,\PSL)$ of Fuchsian representations, there exist a family $(\rho_t)_{t\geq 0}\subset\Hom(\Gamma_g,\PSL)$ of non-Fuchsian representations and, for any $t$ in a small interval $[0,t_0]$, a~$1$-Lip\-schitz, $(j_t,\rho_t)$-equivariant map $\varphi_t : \HH^2\rightarrow\HH^2$, with the following properties:
\begin{itemize}
  \item $\varphi_t$ is an orientation-preserving (\resp orientation-reversing) isometry in restriction to any connected subset of~$\HH^2$ projecting to a union of pants labeled $1$ (\resp $-1$) in $j_t(\Gamma_g)\backslash\HH^2\simeq\Sigma_g$;
  \item for any $\eta>0$ there exists $C<1$ such that $\Lip_p(\varphi_t)\leq C$ for all $t\in [0,t_0]$ and all $p\in\HH^2$ whose image in $j_t(\Gamma_g)\backslash\HH^2\simeq\Sigma_g$ lies inside a pair of pants $P$ labeled~$0$, at distance $\geq\eta$ from the boundary of~$P$.
\end{itemize}
\end{proposition}

Proposition~\ref{prop:unif-prop-signs} is based on the following uniform version of Corollary~\ref{cor:Lip-pants}.

\begin{lemma}\label{lem:pantsfold-uniform}
Let $\Gamma=\langle\alpha,\beta,\gamma \,|\, \alpha\beta\gamma=1\rangle$ be the fundamental group of a pair of pants $\Sigma$, with $\alpha, \beta, \gamma$ corresponding to the boundary loops.
Consider two continuous families $(j_t)_{t\geq 0},(\rho_t)_{t\geq 0}\subset\Hom(\Gamma,\PSL)$ of representations with $j_t$ geometric (Definition~\ref{def:geomrep}), $\rho_t$ nongeometric, and
$$\big(\lambda_{j_t}(\alpha),\lambda_{j_t}(\beta),\lambda_{j_t}(\gamma)\big) = \big(\lambda_{\rho_t}(\alpha),\lambda_{\rho_t}(\beta),\lambda_{\rho_t}(\gamma)\big)$$
for all $t\geq 0$.
Then there exists a family of $1$-Lipschitz, $(j_t,\rho_t)$-equivariant maps $\varphi_t : \HH^2\rightarrow\HH^2$, defined for all $t$ in a small interval $[0,t_0]$, with the following property: for any $\eta>0$ there exists $C<1$ such that $\Lip_p(\varphi_t)\leq C$ for any $t\in [0,t_0]$ and any $p\in\HH^2$ whose image in $j_t(\Gamma)\backslash\HH^2$ lies at distance $\geq\eta$ from the boundary of the convex core.
\end{lemma}

\begin{proof}[Proof of Lemma~\ref{lem:pantsfold-uniform}]
By Corollary~\ref{cor:Lip-pants}, there exists a $1$-Lipschitz, $(j_0,\rho_0)$-equiva\-riant map $f_0 : \HH^2\rightarrow\HH^2$ such that $\Lip_p(f_0)<1$ for any $p\in\HH^2$ whose image in $j_0(\Gamma)\backslash\HH^2$ does not belong to the boundary of the convex core.
If $(j_t,\rho_t)=(j_0,\rho_0)$ for all~$t$, then we may take $\varphi_t=f_0$.
In the general case, we shall build $\varphi_t$ as a small deformation of~$f_0$ in restriction to the preimage of the convex core of $j_t(\Gamma)\backslash\HH^2$.

Choose $\Delta>0$ so that for all small $t\geq 0$, the $2\Delta$-neighborhoods of the boundary components of the convex core of the hyperbolic surface $j_t(\Gamma)\backslash\HH^2$ are disjoint.
Choose a small $\delta\in(0,\Delta/2)$ and let $\sigma_{\delta} : \R^+\rightarrow \R^+$ be the function that satisfies
$$\sigma_{\delta}(\eta) = \left \{ \begin{array}{lll}
0 & \mathrm{for} & 0 \leq \eta \leq 2\delta,\\
\Delta - 2\delta & \mathrm{for} & \eta = \Delta,\\
\eta & \mathrm{for} & \eta \geq 2\Delta 
\end{array} \right .$$
and is affine on $[2\delta,\Delta]$ and $[\Delta,2\Delta]$ (Figure~\ref{fig:fonction}).
\begin{figure}[ht!]
\labellist
\small\hair 2pt
\pinlabel {$0$} at 5 3
\pinlabel {$2\delta$} at 22 3
\pinlabel {$\Delta$} at 45 3
\pinlabel {$2\Delta$} at 83 3
\pinlabel {$\eta$} at 105 16
\pinlabel {$\sigma_{\delta}(\eta)$} at 25 93
\endlabellist
\centering
\includegraphics[width=4cm]{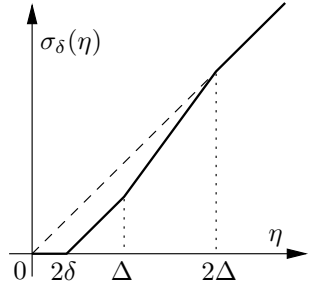}
\caption{The function $\sigma_{\delta}$}
\label{fig:fonction}
\end{figure}
Note that $\sigma_{\delta}$ is $(1+o(1))$-Lipschitz as $\delta\rightarrow 0$, and $1$-Lipschitz away from $[\Delta,2\Delta]$.
For any $t\geq 0$, let $N_t\subset\HH^2$ be the preimage of the convex core of $j_t(\Gamma)\backslash\HH^2$, and let $\pi_t : \HH^2\rightarrow N_t$ be the closest-point projection; it is $1$-Lipschitz.
We set
$$\varphi_0 := f_0 \circ J_{\delta} \circ \pi_0 ,$$
where $J_{\delta}$ is the homotopy of~$\HH^2$ taking any point at distance $\eta\leq 2\Delta$ from a boundary component $\ell_0$ of~$N_0$, to the point at distance $\sigma_{\delta}(\eta)$ from~$\ell_0$ on the same perpendicular ray to~$\ell_0$, leaving other points unchanged.
By construction, in restriction to the $2\delta$-neighborhood of~$\partial N_0$, the map $\varphi_0$ factors through the closest-point projection onto $\partial N_0$.
The function $p\mapsto\Lip_p(f_0)$ is $j_0(\Gamma)$-invariant, upper semicontinuous, and $<1$ on $\HH^2\smallsetminus\partial N_0$, hence bounded away from~$1$ when $p\in N_0$ stays at distance $\geq\Delta-2\delta$ from $\partial N_0$.
This implies that if we have chosen $\delta$ small enough (which we shall assume from now on), then $\Lip(\varphi_0)=1$ and $\Lip_p(\varphi_0)<1$ for all $p$ in the interior of~$N_0$.
For $t>0$, we construct $\varphi_t$ as a deformation of~$\varphi_0$ via a partition of unity, as follows.

Let $\mathcal{U}_t^{\delta}\subset N_t$ be the $\delta$-neighborhood of $\partial N_t$ and $N_t^{\delta}:=N_t\smallsetminus\mathcal{U}_t^{\delta}$ its complement in~$N_t$; we define $\mathcal{U}_t^{2\delta}$ similarly.
Choose a $1$-Lipschitz, $(j_t,\rho_t)$-equivariant map $\varphi_t^0 : \mathcal{U}_t^{2\delta}\rightarrow\nolinebreak\HH^2$ factoring through the closest-point projection onto~$\partial N_t$ and taking any boundary component $\ell_t$ of~$N_t$, stabilized by a cyclic subgroup $j_t(S)$ of $j_t(\Gamma)$, isometrically to the translation axis of $\rho_t(S)$ in~$\HH^2$.
Up to postcomposing each $\varphi_t^0$ with an appropriate shift along the axis of $\rho_t(S)$, we may assume that $\varphi_t^0(p)\rightarrow\varphi_0(p)$ for any $p\in\mathcal{U}_0^{2\delta}$ as $t\rightarrow 0$ (recall that the restriction of $\varphi_0$ to any boundary component of~$N_0$ is an isometry).

Let $B^1,\dots,B^n\subset N_0$ be balls of~$\HH^2$, each projecting injectively to $j_0(\Gamma)\backslash\HH^2$, disjoint from a neighborhood of $\partial N_0$, and such that
$$N_0^{\delta} \subset j_0(\Gamma) \cdot \bigcup_{i=1}^n B^i .$$
For $1\leq i \leq n$, let $\varphi^i_t : j_t(\Gamma)\cdot B^i\rightarrow\HH^2$ be the $(j_t,\rho_t)$-equivariant map that agrees with $\varphi_0$ on~$B^i$.
By construction, for all $1\leq i\leq n$ (\resp for~$i=0$) and for all $p\in j_0(\Gamma)\cdot B^i$ (\resp $p\in\mathcal{U}_0^{2\delta}$) we have $\varphi_t^i(p)\rightarrow\varphi_0(p)$ as $t\rightarrow 0$, uniformly for $p$ in any compact set.
However, the maps $\varphi_t^i$, for $0\leq i\leq n$, may not agree at points where their domains overlap.
The goal is to paste them together by the procedure described in Section~\ref{subsec:rem-Lip}, using a $j_t(\Gamma)$-invariant partition of unity $(\psi^i_t)_{0\leq i\leq n}$ that we now construct.

Let $\psi^0_t:\HH^2\rightarrow [0,1]$ be the function supported on~$\mathcal{U}_t^{2\delta}$ that takes any point at distance $\eta$ from $\partial N_t$ to $\tau(\eta)\in[0,1]$, where $\tau([0,\delta])=1$, where $\tau([2\delta,+\infty))=0$, and where $\tau$ is affine on $[\delta,2\delta]$. 
Let $\psi^1,\dots, \psi^n : \HH^2\rightarrow [0,1]$ be $j_0(\Gamma)$-invariant Lipschitz functions inducing a partition of unity on a neighborhood of~$N_0^{\delta}$, with $\psi^i$ supported in $j_0(\Gamma)\cdot B^i$. 
Since $N_t$ has a compact fundamental domain for $j_t(\Gamma)$ that varies continuously in~$t$ (for instance a right-angled octagon), for small enough~$t$ we have
$$N_t^{\delta} \subset j_t(\Gamma) \cdot \bigcup_{i=1}^n B^i .$$
For $1\leq i\leq n$ and $t\geq 0$, let $\hat{\psi}^i_t : \HH^2\rightarrow [0,1]$ be the $j_t(\Gamma)$-invariant function supported on $j_t(\Gamma)\cdot B^i$ that agrees with $\psi^i$ on~$B^i$.
Then $\sum_{i=1}^n \hat{\psi}^i_t=1+o(1)$ as $t\rightarrow 0$, with an error term uniform on~$N_t^{\delta}$. 
Therefore the functions
$$\psi^0_t \quad\quad\mathrm{and}\quad\quad \psi^i_t := (1-\psi^0_t) \, \frac{\hat{\psi}^i_t}{\sum_{k=1}^n \hat{\psi}^k_t} \ :\ \HH^2 \longrightarrow [0,1]$$
for $1\leq i \leq n$ form a $j_t(\Gamma)$-invariant partition of unity of~$N_t$, subordinated to the covering $\mathcal{U}_t^{2\delta}\cup j_t(\Gamma)\cdot B^1\cup\dots\cup j_t(\Gamma)\cdot B^n\supset N_t$, and are all $L$-Lipschitz for some $L>0$ independent of $i$ and~$t$.

For $t\geq 0$, let $\varphi_t:=\sum_{i=0}^n \psi^i_t\,\varphi^i_t : N_t\rightarrow\HH^2$ be the averaged map defined in Section~\ref{subsec:rem-Lip}: it is $(j_t,\rho_t)$-equivariant by construction.
We extend it to a map $\varphi_t : \HH^2\rightarrow\HH^2$ by precomposing with the closest-point projection $\pi_t :\nolinebreak\HH^2\rightarrow\nolinebreak N_t$.
We claim that the maps $\varphi_t$ satisfy the conclusion of Lemma~\ref{lem:pantsfold-uniform}.
Indeed, by Lemma~\ref{lem:partofunity}, for any $t\geq 0$ and $p$ in the interior of~$N_t$,
\begin{equation}\label{eqn:Leibniz}
\Lip_p(\varphi_t)\leq \sum_{i\in I_t(p)} \big(\Lip_p(\psi_t^i)\,R_t(p) + \psi_t^i(p)\,\Lip_p(\varphi_t^i)\big) ,
\end{equation}
where $I_t(p)$ is the set of indices $0\leq i\leq n$ such that $p$ belongs to the support of~$\psi_t^i$, and $R_t(p)\geq 0$ is the diameter of the set $\{ \varphi_t^i(p)\,|\,i\in I_t(p)\} $.
Let $\eta>0$ be the distance from $p$ to $\partial N_t$.
If $\eta<\delta$, then $\varphi_t$ coincides on a neighborhood of~$p$ with~$\varphi_t^0$, hence with the closest-point projection onto $\partial N_t$ postcomposed with an isometry of~$\HH^2$, and the right-hand side of \eqref{eqn:Leibniz} reduces to
$$\Lip_p(\varphi_t^0) = \frac{1}{\cosh \eta} < 1$$
(see \cite[(A.9)]{gk13} for instance).
If $\eta\geq\delta$, then the bound on $\Lip_p(\varphi_t^0)$~still~holds, and $\Lip_p(\varphi_t^i)$ for $1\leq i\leq n$ can also be uniformly bounded away from~$1$: indeed, $\sup_{q\in B^i}\Lip_q(\varphi_t^i)<1$ since $B^i$ is disjoint from a neighborhood of $\partial N_0$ and the local Lipschitz constant is upper semicontinuous, and we argue by equivariance.
Moreover, all the other contributions to \eqref{eqn:Leibniz} are small: $R_t(p)\rightarrow 0$ as $t\rightarrow 0$, uniformly in~$p$, and $\Lip_p(\psi_t^i)$ is bounded independently of $p,i,t$ (by~$L$).
Therefore, for small $t$ there exists $C<1$, independent of $p$ and~$t$, such that $\Lip_p(\varphi_t)\leq\nolinebreak C$.
This treats the case when $p\in N_t$.
To conclude, we note that on a neighborhood of any $p\in\HH^2\smallsetminus N_t$ the map $\varphi_t$ coincides with the closest-point projection onto $\partial N_t$ postcomposed with an isometry of~$\HH^2$, hence $\Lip_p(\varphi_t)=1/\cosh\eta<1$ where $\eta=d(p,\partial N_t)$.
\end{proof}

\begin{proof}[Proof of Proposition~\ref{prop:unif-prop-signs}]
Let $\Upsilon$ be a lamination of~$\Sigma_g$ consisting of all the cuffs of~$\Pi$ together with a triskelion lamination inside each pair of pants labeled~$0$.
Let $c : \Sigma_g\smallsetminus\Upsilon\rightarrow\{ -1,1\}$ be a coloring taking the value $1$ (\resp $-1$) on each pair of pants labeled $1$ (\resp $-1$), and both values on each pair of pants labeled~$0$.
For any $t\geq\nolinebreak 0$, let $\rho'_t$ be the folding of $j_t$ along~$\Upsilon$ with coloring~$c$.
We now argue similarly to the proof of Proposition~\ref{prop:signs} in Section~\ref{subsec:provesigns}: for each pair of pants $P$ in~$\Pi$, choose a subgroup $\Gamma^P$ of~$\Gamma_g$ which is conjugate to $\pi_1(P)$, and for any $t\geq 0$ a lift $\widetilde{P}_t\subset\HH^2$ of the convex core of $j_t(\Gamma^P)\backslash\HH^2$.
If $P$ is labeled $1$ (\resp $-1$), then for any $t\geq 0$ the restrictions of $j_t$ and $\rho'_t$ to~$\Gamma_P$ are conjugate by some orientation-preserving (\resp orientation-reversing) isometry $\varphi_t^P$ of~$\HH^2$.
If $P$ is labeled~$0$, then by Lemma~\ref{lem:pantsfold-uniform} there is a family of $1$-Lipschitz, $(j_t|_{\Gamma^P},\rho'_t|_{\Gamma^P})$-equivariant maps $\varphi_t^P : \HH^2\rightarrow\HH^2$, defined for all $t$ in a small interval $[0,t_0]$, with the following property: for any $\eta>0$ there exists $C<1$ such that $\Lip_p(\varphi_t^P)\leq C$ for all $t\in [0,t_0]$ and all $p\in\widetilde{P}_t$ at distance $\geq\eta$ from~$\partial\widetilde{P}_t$.
The collection of all maps~$\varphi_t^P$, extended $(j_t,\rho'_t)$-equivariantly, piece together to yield a map $\varphi_t^{\ast} : \HH^2\smallsetminus\mathcal{C}_t\rightarrow\HH^2$, where $\mathcal{C}_t$ is the union of all geodesics of~$\HH^2$ projecting to cuffs of~$\Pi$ in $j_t(\Gamma_g)\backslash\HH^2\simeq\Sigma_g$.
The obstruction to extending $\varphi_t^{\ast}$ by continuity on each geodesic $\ell_t\subset\mathcal{C}_t$ is that the maps on either side of~$\ell_t$ may disagree by a constant shift along~$\ell_t$ if $\ell_t$ separates two pairs of pants labeled $(\pm 1,0)$ or $(0,0)$.
This discrepancy $\delta(\ell_t)\in\R$ is the same on the whole $j_t(\Gamma_g)$-orbit of~$\ell_t$.
To correct it, we precompose the folding $\rho'_t$ of~$j_t$ with an earthquake supported on the cuff associated with~$\ell_t$ (in the $j_t$-metric), of length $-\delta(\ell_t)$.
We repeat for each $j_t(\Gamma_g)$-orbit in~$\mathcal{C}_t$, and eventually obtain a new folded representation~$\rho_t$.
By construction, there is a family of $1$-Lipschitz, $(j_t,\rho_t)$-equivariant maps $\varphi_t : \HH^2\rightarrow\HH^2$ satisfying Proposition~\ref{prop:unif-prop-signs}, obtained simply by gluing together isometric translates of the~$\varphi_t^P$.
\end{proof}

\section{Surjectivity of the two projections}\label{sec:surj}

In this section we prove Theorem~\ref{thm:domin}.
We first construct uniformly lengthening deformations of surfaces with boundary (Section~\ref{subsec:boundary}), then glue these together according to combinatorics given by Proposition~\ref{prop:signs} (Sections \ref{subsec:hallali1} and~\ref{subsec:hallali2}).
Section~\ref{subsec:proof-lemma} is devoted to the proof of a technical lemma.

\subsection{Uniformly lengthening deformations of compact hyperbolic surfaces with boundary}\label{subsec:boundary}

Our two main tools to prove Theorem~\ref{thm:domin} are Proposition~\ref{prop:signs} and the following lemma.

\begin{lemma}\label{lem:surfwithbound}
Let $\Gamma$ be the fundamental group and $j_0\in\Hom(\Gamma,\PSL)$ the holonomy of a compact, connected, hyperbolic surface $\Sigma$ with nonempty geodesic boundary.
Then there exist $t_0>0$ and a continuous family of representations $(j_t)_{0\leq t \leq t_0}$ with the following properties:
\begin{itemize}
  \item[$(a)$] $\lambda_{j_0}(\gamma)=(1-t)\,\lambda_{j_t}(\gamma)$ for any $t\in [0,t_0]$ and any $\gamma\in\Gamma$ corresponding to a boundary component of~$\Sigma$;
  \item[$(b)$] $\sup_{\gamma\in\Gamma\smallsetminus \{1\}}\, \frac{\lambda_{j_0}(\gamma)}{\lambda_{j_t}(\gamma)}<1$ for any $t\in (0,t_0]$;
  \item[$(c)$] $j_t(\gamma)=j_0(\gamma)+O(t)$ for any $\gamma\in\Gamma$ as $t\rightarrow 0$, where both sides are seen as $2\times 2$ real matrices with determinant~$1$;
  \item[$(d)$] for any compact subset $K$ of~$\HH^2$ projecting to the interior of the convex core of $j_0(\Gamma)\backslash\HH^2$, there exists $L>0$ such that
  $$d(p,f_t(p))\leq Lt$$
  for any $p\in K$, any $t\in[0,t_0]$, and any $1$-Lipschitz, $(j_t,j_0)$-equivariant map $f_t : \HH^2\rightarrow\HH^2$.
\end{itemize}
\end{lemma}

As in Section~\ref{subsec:pants}, the convex core of $j_0(\Gamma)\backslash\HH^2$ naturally identifies with~$\Sigma$.
The idea is to construct the representations $j_t$ as holonomies of hyperbolic surfaces obtained from $j_0(\Gamma)\backslash\HH^2$ by \emph{strip deformations}.
This type of deformation was first introduced by Thurston \cite[proof of Lem.\,3.4]{thu86}; we refer to \cite{pt10} and \cite{dgk13} for more details.

\begin{proof}
We first explain how to lengthen one boundary component $\beta$ of~$\Sigma$.
Choose a finite collection of disjoint, biinfinite geodesic arcs $\alpha_1,\dots, \alpha_n \subset j_0(\Gamma)\backslash\HH^2$, each crossing $\beta$ orthogonally twice, and subdividing the convex core $\Sigma$ into right-angled hexagons and one-holed right-angled bigons.
Along each arc~$\alpha_i$, following \cite{thu86}, slice $j_0(\Gamma)\backslash\HH^2$ open and insert a strip $A_i$ of~$\HH^2$, bounded by two geodesics, with narrowest cross section at the midpoint of $\alpha_i\cap\Sigma$ (see Figure~\ref{fig:deformation}).
This yields a new complete hyperbolic surface, with a compact convex core, equipped with a natural $1$-Lipschitz map $\varsigma_t^{\beta}$ to $j_0(\Gamma)\backslash\HH^2$ obtained by collapsing the strips $A_i$ back to lines.
Note that the image under~$\varsigma_t^{\beta}$ of the new convex core is \emph{strictly contained} in~$\Sigma$ (see Figure~\ref{fig:deformation}).
The geodesic corresponding to~$\beta$ is longer in the new surface than in~$\Sigma$; by adjusting the widths of the strips~$A_i$, we may assume that the ratio of lengths is $\frac{1}{1-t}$.
Note that the appropriate widths for this ratio are in $O(t)$ as $t\rightarrow 0$.
All lengths of geodesics corresponding to boundary components other than~$\beta$ are unchanged.

\begin{figure}[ht!]
\labellist
\small\hair 2pt
\pinlabel {$\alpha_i$} at 25 48
\pinlabel {$\beta$} at 12 12
\pinlabel {$\Sigma$} at 72 68
\pinlabel {$\varsigma_t^{\beta}$} at 130 10
\pinlabel {$A_i$} at 170 48
\endlabellist
\centering
\includegraphics[width=11cm]{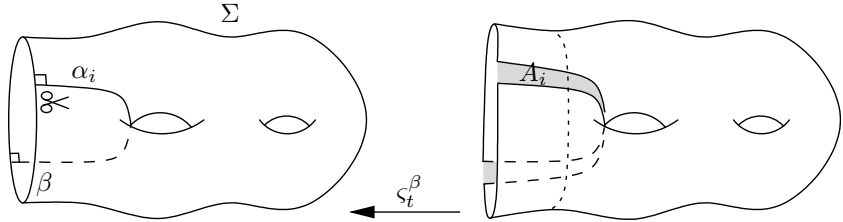}
\caption{A strip deformation. In the source of the collapsing map $\varsigma_t^{\beta}$ we show the new peripheral geodesic, dotted.}
\label{fig:deformation}
\end{figure}

Repeat the construction, iteratively, for all boundary components $\beta_1,\dots, \beta_r$ of~$\Sigma$, in some arbitrary order: we thus obtain a new complete hyperbolic surface $j_t(\Gamma)\backslash\HH^2$, with a compact convex core~$\Sigma_t$, such that $j_t$ satisfies~(a).
We claim that $j_t$ also satisfies~(b).
Indeed, consider the $1$-Lipschitz map $\varsigma_t:=\varsigma^{\beta_r}_t\circ\dots\circ\varsigma^{\beta_1}_t$ from $\Sigma_t$ to~$\Sigma$.
If $1$ were its optimal Lipschitz constant, then by Lemma~\ref{lem:stretchlocus} there would exist a geodesic lamination of~$\Sigma_t$ whose leaves are isometrically preserved by~$\varsigma_t$.
But this is not the case here since for every~$i$, the map $\varsigma^{\beta_i}_t$ does not isometrically preserve any geodesic lamination except the boundary components other than~$\beta_i$.
Therefore $\varsigma_t$ has Lipschitz constant $<1$, which implies~(b) by Remark~\ref{rem:Clambda-CLip}.

Up to replacing each $j_t$ with a conjugate under $\PSL$, we may assume that (c) holds.
Indeed, it is well known that there exist elements $\gamma_1,\dots,\gamma_n\in\Gamma$ whose length functions form a smooth coordinate system for $\Hom(\Gamma,\PSL)/\PSL$ near~$[j_0]$ (see \cite[Th.\,2.1]{gx11} for instance).
For any~$i$, the preimage under~$\varsigma_t$ of the closed geodesic of~$\Sigma$ associated with~$\gamma_i$ is obtained by expanding finitely many strips of width $O(t)$, hence
$\lambda_{j_t}(\gamma_i)\leq\lambda_{j_0}(\gamma_i)+O(t)$
as $t\rightarrow 0$.
On the other hand, $\lambda_{j_t}(\gamma_i)\geq\lambda_{j_0}(\gamma_i)$ due to the existence of the $1$-Lipschitz map~$\varsigma_t$.
Therefore, $d'(j_0,j_t)=O(t)$ for any smooth metric $d'$ on a neighborhood of $[j_0]$ in $\Hom(\Gamma,\PSL)/\PSL$.

To check~(d), we use a perturbative version of the argument that a $j_0(\Gamma)$-invariant, $1$-Lipschitz map must be the identity on the preimage $N_0\subset\HH^2$ of the convex core $\Sigma$ of $j_0(\Gamma)\backslash\HH^2$.
For any hyperbolic element $h\in\PSL$, with translation axis $\mathcal{A}_h\subset\HH^2$, and for any $p\in\HH^2$, a classical formula gives
\begin{equation}\label{eqn:trigo}
\sinh\Big(\frac{d(p,h\cdot p)}{2}\Big) = \sinh\Big(\frac{\lambda(h)}{2}\Big) \cdot \cosh d(p,\mathcal{A}_h)
\end{equation}
(see Figure~\ref{fig:quadrilatere}, left).
Consider $p\in\HH^2$ in the interior of~$N_0$.
We can find three translation axes $\mathcal{A}_{j_0(\gamma_1)},\mathcal{A}_{j_0(\gamma_2)},\mathcal{A}_{j_0(\gamma_3)}\subset\partial N_0$ of elements of $j_0(\Gamma)$ such that if $q_i$ denotes the projection of $p$ to~$\mathcal{A}_{j_0(\gamma_i)}$, then $p$ belongs to the interior of the triangle $q_1 q_2 q_3$.
For any $t\geq 0$ and any $1$-Lipschitz, $(j_t,j_0)$-equivariant map $f_t : \HH^2\rightarrow\HH^2$,
$$d\big(f_t(p),j_0(\gamma_i)\cdot f_t(p)\big) \leq d(p,j_t(\gamma_i)\cdot p) ,$$
which by \eqref{eqn:trigo} may be written as
$$\sinh\Big(\frac{\lambda_{j_0}(\gamma_i)}{2}\Big) \cdot \cosh d\big(f_t(p),\mathcal{A}_{j_0(\gamma_i)}\big) \leq \sinh\Big(\frac{\lambda_{j_t}(\gamma_i)}{2}\Big) \cdot \cosh d\big(p,\mathcal{A}_{j_t(\gamma_i)}\big) .$$
Since $\lambda_{j_0}(\gamma_i)=\lambda_{j_t}(\gamma_i)+O(t)$ and $d(p, \mathcal{A}_{j_t(\gamma_i)})=d(p, \mathcal{A}_{j_0(\gamma_i)})+O(t)$ by~(c), this implies
$$\cosh d\big(f_t(p),\mathcal{A}_{j_0(\gamma_i)}\big) \leq \cosh d\big(p,\mathcal{A}_{j_0(\gamma_i)}\big) + O(t) ,$$
where the error term does not depend on the choice of the map~$f_t$.
Since $d(p, \mathcal{A}_{j_0(\gamma_i)})>0$, we may invert the hyperbolic cosine:
$$d\big(f_t(p),\mathcal{A}_{j_0(\gamma_i)}\big) \leq d\big(p,\mathcal{A}_{j_0(\gamma_i)}\big) + O(t) .$$
Applied to $i=1,2,3$, this means that $f_t(p)$ belongs to a curvilinear triangle around~$p$ bounded by three hypercycles (curves at constant distance from a geodesic line) expanding at rate $O(t)$ as $t$ becomes positive, hence $d(p,f_t(p))=O(t)$ (see Figure~\ref{fig:quadrilatere}, right).
All estimates $O(t)$ are robust under small perturbations of $p$, hence can be made uniform (and still independent of $f_t$) for $p$ in a compact set~$K$, yielding~(d).
\end{proof}

\begin{figure}[ht!]
\labellist
\small\hair 2pt
\pinlabel {$\lambda(h)$} at 60 17
\pinlabel {$p$} at 3 59
\pinlabel {$h\cdot p$} at 142 60
\pinlabel {$\mathcal{A}_h$} at -5 10
\pinlabel {$d(p,\mathcal{A}_h)$} at 1 30
\pinlabel {$p$} at 210 38
\pinlabel {$\mathcal{A}_{j_0(\gamma_1)}$} at 212 68
\pinlabel {$\mathcal{A}_{j_0(\gamma_2)}$} at 170 30
\pinlabel {$\mathcal{A}_{j_0(\gamma_3)}$} at 234 20
\endlabellist
\centering
\includegraphics[width=11cm]{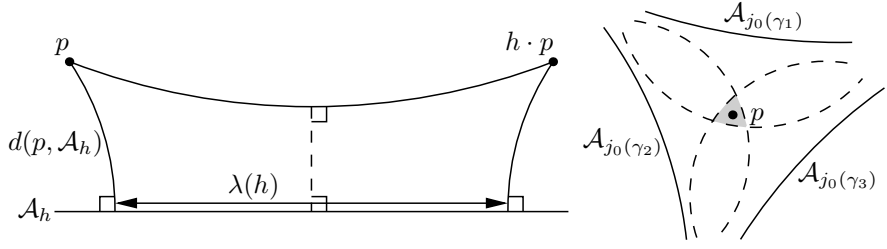}
\caption{Left: A hyperbolic quadrilateral with two right angles. Right: The point $f_t(p)$ belongs to the shaded region.}
\label{fig:quadrilatere}
\end{figure}

\subsection{Gluing surfaces with boundary}\label{subsec:hallali1}

We now prove the first statement of Theorem~\ref{thm:domin}: namely, given $[\rho]\in\Repnfd$, we construct $[j]\in\Repfd$ that strictly dominates~$[\rho]$.

If $\lambda_{\rho}\equiv 0$, then any $[j]\in\Repfd$ strictly dominates~$[\rho]$.
We now suppose that $\lambda_{\rho}\not\equiv 0$.
Proposition~\ref{prop:signs}.(1) then gives us an element $[j_0]\in\Repfd$, a labeled pants decomposition $\Pi$ of~$\Sigma_g$, and, for any $j_0,\rho\in\Hom(\Gamma_g,\PSL)$ in the respective classes $[j_0],[\rho]$ (which we now fix), a $1$-Lipschitz, $(j_0,\rho)$-equivariant map $f :\nolinebreak\HH^2\rightarrow\nolinebreak\HH^2$ that is an orientation-preserving (\resp orien\-tation-reversing) isometry in restriction to any connected subset of~$\HH^2$ projecting to a union of pants labeled $1$ (\resp $-1$) in $j_0(\Gamma_g)\backslash\HH^2\simeq\Sigma_g$, and that satisfies $\Lip_p(f)<1$ for any $p\in\HH^2$ projecting to the interior of a pair of pants labeled~$0$.
Not all pairs of pants are labeled $1$, and not all~$-1$, since $j_0$ and $\rho$ are not conjugate under $\PGL$.
By Remark~\ref{rem:Clambda-CLip}, the class $[j_0]$ dominates $[\rho]$ in the sense that $\lambda(\rho(\gamma))\leq\lambda(j_0(\gamma))$ for all $\gamma\in\Gamma_g$.
Our goal is to use Lemma~\ref{lem:surfwithbound} to modify $j_0$ into a representation $j$ such that $[j]$ \emph{strictly} dominates~$[\rho]$.
For this purpose, we erase all the cuffs that separate two pairs of pants of~$\Pi$ with labels $(1,1)$ or $(-1,-1)$, and write
$$\Sigma_g = \Sigma^{1} \cup \dots \cup \Sigma^{m} ,$$
where $\Sigma^{i}$, for any $1\leq i\leq m$, is a compact surface with boundary that is
\begin{itemize}
  \item either a pair of pants labeled~$0$,
  \item or a full connected component of the subsurface of $\Sigma_g$ made of pants labeled $1$,
  \item or a full connected component of the subsurface of $\Sigma_g$ made of pants labeled $-1$
\end{itemize}
(see Figure~\ref{fig:decomposition}).
The boundary components of the~$\Sigma^{i}$ are the cuffs that separated two pairs of pants of~$\Pi$ with labels $(1,-1)$, $(\pm 1,0)$, or $(0,0)$.

\begin{figure}[ht!]
\labellist
\small\hair 2pt
\pinlabel {$1$} at 40 15
\pinlabel {$1$} at 85 15
\pinlabel {$-1$} at 113 15
\pinlabel {$0$} at 142 15
\pinlabel {$0$} at 168 15
\pinlabel {$1$} at 208 15
\endlabellist
\centering
\includegraphics[width=8cm]{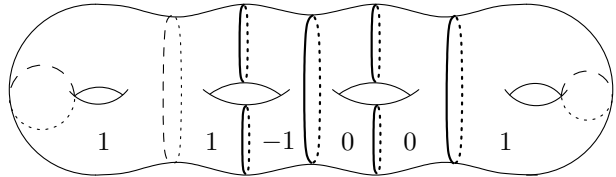}
\caption{A labeled pants decomposition with $m=5$. The boundary components of the $\Sigma^{i}$, $1\leq i\leq 5$, are in bold.}
\label{fig:decomposition}
\end{figure}

Choose a small $\delta>0$ such that in all hyperbolic metrics on~$\Sigma_g$ which are close enough to that defined by~$j_0$, any simple geodesic entering the $\delta$-neighborhood of the geodesic representative of a cuff of~$\Pi$ crosses it.
Let $\mathcal{C}_0\subset\HH^2$ be the union of all geodesic lines of~$\HH^2$ projecting to boundary components of the $\Sigma^{i}$ in $j_0(\Gamma_g)\backslash\HH^2\simeq\Sigma_g$, let $N_0^{\delta}\subset\HH^2$ be the complement of the $\delta$-neighborhood of~$\mathcal{C}_0$, and let $K\subset\HH^2\smallsetminus\mathcal{C}_0$ be a compact set whose interior contains a fundamental domain of~$N_0^{\delta}$ for the action of~$j_0(\Gamma_g)$, with $m$ connected components projecting respectively to $\Sigma^{1},\dots,\Sigma^{m}$.
We apply Lemma~\ref{lem:surfwithbound} to $\Gamma^{i}:=\pi_1(\Sigma^{i})$ and $j_0^{i}:=j_0|_{\Gamma^{i}}$ and obtain continuous families $(j_t^{i})_{0\leq t \leq t_0}\subset\Hom(\Gamma^{i},\PSL)$ of representations, for $1\leq i\leq m$, satisfying properties (a),(b),(c),(d) of Lemma~\ref{lem:surfwithbound}, with a uniform constant $L>0$ for the compact set $K\subset\HH^2\smallsetminus\mathcal{C}_0$.
For any $t\in [0,t_0]$, using~(a), we can glue together the (compact) convex cores of the $j_t^{i}(\Gamma^{i})\backslash\HH^2$ following the same combinatorics as the~$\Sigma^{i}$: this gives a closed hyperbolic surface of genus~$g$, hence a holonomy representation $j_t\in\Hom(\Gamma_g,\PSL)$.
By~(c), up to adjusting the twist parameters, we may assume that
\begin{equation}\label{eqn:jt-j0-close-boundary}
j_t(\gamma) = j_0(\gamma) + O(t)
\end{equation}
for any $\gamma\in\Gamma_g$ as $t\rightarrow 0$, where both sides are seen as $2\times 2$ real matrices with determinant~$1$.
To complete the proof of the first statement of Theorem~\ref{thm:domin}, it is sufficient to prove that for small enough $t>0$,
\begin{equation}\label{eqn:lambda-rho-j_t}
\sup_{\gamma\in (\Gamma_g)_s}\ \frac{\lambda_{\rho}(\gamma)}{\lambda_{j_t}(\gamma)} < 1 ,
\end{equation}
where $(\Gamma_g)_s$ is the set of nontrivial elements of~$\Gamma_g$ corresponding to simple closed curves on~$\Sigma_g$: then $[j]:=[j_t]$ will strictly dominate~$[\rho]$ by Theorem~\ref{thm:contcont}.
Note that $\lambda(j_t(\gamma)) = \lambda(j_t^{i}(\gamma))$ for all $\gamma$ in $\Gamma^{i}$, seen as a subgroup of~$\Gamma_g$; thus (b) gives the control required in \eqref{eqn:lambda-rho-j_t} for simple closed curves \emph{contained in one of the~$\Sigma^{i}$}.
We now explain why the lengths of the other simple closed curves also decrease uniformly, based on (c) and~(d).

For any $t\in (0,t_0]$, let $\mathcal{C}_t\subset\HH^2$ be the union of the lifts to~$\HH^2$ of the simple closed geodesics of $j_t(\Gamma_g)\backslash\HH^2\simeq\Sigma_g$ corresponding to~$\mathcal{C}_0$ and let $N_t^{\delta}$ be the complement of the $\delta$-neighborhood of $\mathcal{C}_t$ in~$\HH^2$.
For $t$ small enough, we can find a fundamental domain $K_t$ of $N_t^{\delta}$ for the action of $j_t(\Gamma_g)$ that is contained in~$K$ and has $m$ connected components.
By (b) and Theorem~\ref{thm:contcont}, for any $1\leq i\leq m$ and $t\in (0,t_0]$ there exists a $(j_t|_{\Gamma^{i}},j_0|_{\Gamma^{i}})$-equivariant map $f_t^{i} : \HH^2\rightarrow\HH^2$ with $\Lip(f_t^{i})<1$.
For small $t>0$, we choose a $(j_t,j_0)$-equivariant map $f_t : (N_t^{\delta}\cup\mathcal{C}_t)\rightarrow\HH^2$ such that
\begin{itemize} 
  \item $f_t=f_t^{i}$ on the component of $K_t$ projecting to~$\Sigma^{i}$, for all $1\leq i\leq m$;
  \item $f_t$ takes any geodesic line in~$\mathcal{C}_t$ to the corresponding line in~$\mathcal{C}_0$, multiplying all distances on it by the uniform factor $(1-t)$.
\end{itemize}
We choose the $f_t$ so that, in addition, for any compact set $K'\subset\HH^2$ there exists $L_1\geq 0$ such that $d(x',f_t(x'))\leq L_1t$ for all small enough $t>0$ and all $x'\in\mathcal{C}_t\cap K'$.
Consider the $(j_t,\rho)$-equivariant map
$$F_t := f \circ f_t :\ (N_t^{\delta} \cup \mathcal{C}_t) \longrightarrow \HH^2 ,$$
where $f : \HH^2\rightarrow\HH^2$ is the $(j_0,\rho)$-equivariant map from the beginning of the proof.
In order to prove \eqref{eqn:lambda-rho-j_t}, it is sufficient to establish the following.

\begin{lemma}\label{lem:hallali}
For small enough $t>0$, there exists $C<1$ such that for all $p,q\in\partial N_t^{\delta}$ lying at distance $\delta$ from a line $\ell_t\subset\mathcal{C}_t$, on opposite sides of~$\ell_t$,
$$d(F_t(p ),F_t(q)) \leq C \, d(p,q).$$
\end{lemma}

\noindent
Indeed, fix a small $t>0$.
Any geodesic segment $I=[p,q]$ of~$\HH^2$ projecting to a closed geodesic of $j_t(\Gamma_g)\backslash\HH^2\simeq\Sigma_g$ may be decomposed into subsegments $I_1,\dots,I_n$ contained in~$N_t^{\delta}$ alternating with subsegments $I'_1,\dots,I'_n$ crossing connected components of $\HH^2\smallsetminus N_t^{\delta}$ (indeed, any simple closed curve that enters one of these components crosses it, by choice of~$\delta$).
By construction, the map $F_t$ has Lipschitz constant $<1$ on each connected component of~$N_t^{\delta}$, hence moves the endpoints of each~$I_k$ closer together by a uniform factor (independent of~$I$).
Lemma~\ref{lem:hallali} ensures that the same holds for the~$I'_k$.
Thus the ratio $d(F_t(p),F_t(q))/d(p,q)$ is bounded by some factor $C'<1$ independent of~$I$, and the corresponding element $\gamma\in\Gamma_g$ satisfies $\lambda(\rho(\gamma))\leq\nolinebreak C'\lambda(j_t(\gamma))$.
This proves \eqref{eqn:lambda-rho-j_t}, hence completes the proof of the first statement of Theorem~\ref{thm:domin}.

\subsection{Proof of Lemma~\ref{lem:hallali}}\label{subsec:proof-lemma}

In this section we give a proof of Lemma~\ref{lem:hallali}.
We first make the following observation.

\begin{observation}\label{obs:additive}
There exists $L'\geq 0$ such that for any small enough $t>0$, any $p\in\partial N_t^{\delta}$ at distance $\delta$ from a geodesic $\ell_t\subset\mathcal{C}_t$, and any $x\in\ell_t$,
$$d(f_t(p),f_t(x)) \leq (1-t) \, d(p,x) + L' t .$$
\end{observation}

\begin{proof}[Proof of Observation~\ref{obs:additive}]
Since $f_t$ is $(j_t,j_0)$-equivariant and $\mathcal{C}_0$ has only fini\-tely many connected components modulo $j_0(\Gamma_g)$, we may fix a geodesic $\ell_0\subset\mathcal{C}_0$ and prove the observation only for the geodesics $\ell_t\subset\mathcal{C}_t$ corresponding to~$\ell_0$.
For any $t>0$, the map $f_t$ takes $\ell_t$ linearly to~$\ell_0$, multiplying all distances by the uniform factor $1-t$.
Let $h_t : \HH^2\rightarrow\HH^2$ be the orientation-preserving map that coincides with $f_t$ on~$\ell_t$, takes any line orthogonal to~$\ell_t$ to a line orthogonal to~$\ell_0$, and multiplies all distances by $1-t$ on such lines.
At distance $\eta$ from~$\ell_t$, the differential of $h_t$ has principal values $1-t$ and $(1-t)\cosh((1-t)\eta)/\cosh \eta \leq 1-t$ (see \cite[\!\!(A.9)]{gk13}), hence $\Lip(h_t)\leq 1-t$~and
$$d(f_t(x),h_t(p)) = d(h_t(x),h_t(p)) \leq (1-t) \, d(x,p)$$
for all $x\in\ell_t$ and $p\in\HH^2$.
By the triangle inequality, it is enough to find $L'\geq\nolinebreak 0$ such that $d(h_t(p),f_t(p))\leq L't$ for all small enough $t>0$ and all $p\in\nolinebreak\partial N_t^{\delta}$ at distance $\delta$ from~$\ell_t$.
Since $f_t$ and~$h_t$ are both $(j_t,j_0)$-equivariant under the stabilizer $S$ of $\ell_0$ in~$\Gamma_g$, and $j_t(S)$ acts cocompactly on the set $\overline{\mathcal{U}}_t$ of points at distance $\leq\delta$ from~$\ell_t$, we may restrict to $p$ in a compact fundamental domain of $\overline{\mathcal{U}}_t$ for $j_t(S)$.
Let $K'\subset\HH^2$ be a compact set containing such fundamental domains for all $t\in [0,t_0]$.
By construction of~$f_t$, there exists $L_1\geq 0$ such that $d(x',f_t(x'))\leq L_1t$ for all small enough $t>0$ and all $x'\in\ell_t\cap K'$.
By definition of~$h_t$, this implies the existence of $L_2\geq 0$ such that $d(p,h_t(p))\leq L_2t$ for all small enough $t>0$ and all $p\in K'$.
On the other hand, condition (d) of Lemma~\ref{lem:surfwithbound} (applied to the $\Gamma^i$ and $j_0^i$ as in Section~\ref{subsec:hallali1}) implies the existence of $L_3\geq 0$ such that $d(p,f_t(p))\leq L_3t$ for all $t$ and $p\in\partial N_t^{\delta}\cap K'$.
By the triangle inequality, we may take $L'=L_2+L_3$.
\end{proof}

\begin{proof}[Proof of Lemma~\ref{lem:hallali}]
As in the proof of Observation~\ref{obs:additive}, we may fix a geodesic $\ell_0\subset\mathcal{C}_0$ and restrict to the geodesics $\ell_t\subset\mathcal{C}_t$ corresponding to~$\ell_0$.
Fix a small $t>0$ and consider $p,q\in\partial N_t^{\delta}$ lying at distance $\delta$ from~$\ell_t$, on opposite sides of~$\ell_t$.
The segment $[p,q]$ can be subdivided, at its intersection point $x$ with~$\ell_t$, into two subsegments to which Observation~\ref{obs:additive} applies, yielding
\begin{equation}\label{eqn:splice}
\left \{ \begin{array}{rcl}
d(f_t(p), f_t(x)) & \leq & (1-t) \, d(p,x) + L't ,\\
d(f_t(x), f_t(q)) & \leq & (1-t) \, d(x,q) + L't .
\end{array} \right .
\end{equation}
Up to switching $p$ and~$q$, we may assume that either $[p,x]$ projects to a pair of pants labeled~$0$ in $j_t(\Gamma_g)\backslash\HH^2\simeq\Sigma_g$, or $[p,x]$ projects to a pair of pants labeled~$1$ and $[x,q]$ to a pair of pants labeled $-1$.

Suppose that $[p,x]$ projects to a pair of pants labeled~$0$ in $j_t(\Gamma_g)\backslash\HH^2\simeq\Sigma_g$.
We first observe that if $t$ is small enough (independently of~$p$), then
\begin{equation}\label{eqn:f_t(p)-far}
d(f_t(p),\ell_0)\geq \frac{3\delta}{4} .
\end{equation}
Indeed, as in the proof of Observation~\ref{obs:additive}, the inequality is true for $p\in\partial N_t^{\delta}$ in a fixed compact set~$K'$ independent of~$t$, by condition (d) of Lemma~\ref{lem:surfwithbound} and \eqref{eqn:jt-j0-close-boundary}, and we then use the fact that $f_t$ is $(j_t,j_0)$-equivariant under the stabilizer $S$ of $\ell_0$ in~$\Gamma_g$, which acts cocompactly (by~$j_t$) on the set of points at distance $\delta$ from~$\ell_t$.
By \eqref{eqn:f_t(p)-far}, if $t$ is small enough (independently~of~$p$), then the segment $[f_t(p),f_t(x)]$ spends at least $\delta/4$ units of length in the complement $N_0^{\delta/2}$ of the $\delta/2$-neighborhood of~$\mathcal{C}_0$.
The point is that $\Lip_y(f)<\nolinebreak 1$ for all $y\in\HH^2\smallsetminus\mathcal{C}_0$ projecting to a pair of pants labeled~$0$ in $j_0(\Gamma_g)\backslash\HH^2\simeq\Sigma_g$, and this bound is uniform in restriction to~$N_0^{\delta/2}$ since the function $p\mapsto\Lip_p(f)$ is upper semicontinuous and $j_0(\Gamma_g)$-invariant.
Remark~\ref{rem:local-Lip} thus implies the existence of a constant $\varepsilon>0$, independent of $t,\ell_t,p,x$, such that
\begin{equation}\label{eqn:rabiot}
d\big(f\circ f_t(p), f\circ f_t(x)\big) \leq d(f_t(p), f_t(x)) - \varepsilon .
\end{equation}
Using the triangle inequality and the fact that $f$ is $1$-Lipschitz, together with \eqref{eqn:splice} and \eqref{eqn:rabiot}, we find
\begin{eqnarray*}
d(F_t(p), F_t(q)) & \leq & d\big(f\circ f_t(p), f\circ f_t(x)\big) + d\big(f\circ f_t(x), f\circ f_t(q)\big) \\
& \leq & (1-t)\,d(p,x) + L't -\varepsilon + (1-t)\,d(x,q) + L' t ,
\end{eqnarray*} 
which is bounded by $(1-t)\,d(p,q)$ as soon as $t\leq\varepsilon/(2L')$.

Suppose that $[p,x]$ projects to a pair of pants labeled~$1$ and $[x,q]$ to a pair of pants labeled $-1$.
We then use the fact that the continuous map $f$ folds along $\ell_0=f_t(\ell_t)$: in restriction to the connected component of $\HH^2\smallsetminus\mathcal{C}_0$ containing $f_t(p)$ (\resp $f_t(q)$), it is an isometry preserving (\resp reversing) the orientation.
In particular, $d(F_t( p),F_p(q))<d(f_t( p),f_t(q))$.
Moreover,~this inequality can be made uniform in the following sense: there exists $\varepsilon>0$ such that
$$d(F_t(p),F_t(q)) \leq d(f_t(p),f_t(q)) - \varepsilon$$
whenever $f_t(p)$ and~$f_t(q)$ lie at distance $\geq 3\delta/4$ from~$\ell_0$ (which is the~case for $t$ small enough by \eqref{eqn:f_t(p)-far}) and at distance $\leq 3L'$ from each other.
By \eqref{eqn:splice},
\begin{equation}\label{eqn:pivot}
d(f_t(p),f_t(q)) \leq (1-t) \, d(p,q) + 2 L' t ,
\end{equation} 
which implies
$$d(F_t(p),F_t(q)) \leq (1-t)\,d(p,q)$$
for $d(p,q)\leq 3L'$ as soon as $t\leq\varepsilon/(2L')$ is small enough.
If $d(p,q)\geq 3L'$, then applying the $1$-Lipschitz map $f$ to \eqref{eqn:pivot} directly gives
$$d(F_t(p),F_t(q)) \leq (1-t) \, d(p,q) + 2L' t \leq \Big(1-\frac{t}{3}\Big) \, d(p,q) .\qedhere$$
\end{proof}

\subsection{Folding a given surface} \label{subsec:hallali2}

We now prove the second statement of Theorem~\ref{thm:domin}: namely, given $[j_0]\in\Repfd$ and an integer $k\in (-2g+2,2g-2)$, we construct $[\rho]\in\Repnfd$ with $\mathrm{eu}(\rho)=k$ that is strictly dominated by~$[j_0]$.

It is easy to find $[\rho]$ with $\mathrm{eu}(\rho)=k$ such that $\lambda_{\rho}(\gamma)\leq\lambda_{j_0}(\gamma)$ for all $\gamma\in\Gamma_g$: just decompose $\Sigma_g$ into pairs of pants and attribute arbitrary values $0,1,-1$ to each so that the sum is~$k$.
Consider a lamination $\Upsilon$ of~$\Sigma_g$ consisting of all the cuffs together with a triskelion lamination inside each pair of pants labeled~$0$, and let $c : \Sigma_g\smallsetminus\Upsilon\rightarrow\{ -1,1\}$ be a coloring taking the value $1$ (\resp $-1$) on each pair of pants labeled $1$ (\resp $-1$), and both values on each pair of pants labeled~$0$.
Folding along~$\Upsilon$ with the coloring~$c$ gives an element $[\rho]\in\Repnfd$ with $\lambda_{\rho}(\gamma)\leq\lambda_{j_0}(\gamma)$ for all $\gamma\in\Gamma_g$.
However, we need a \emph{strict} domination: the idea is to obtain $\rho$ by folding, not $j_0$, but a small deformation of $j_0$.
For this purpose, we use the following result, which is analogous to Lemma~\ref{lem:surfwithbound}.

\begin{lemma}\label{lem:surfwithbound2}
Let $\Gamma$ be the fundamental group and $j_0\in\Hom(\Gamma,\PSL)$ the holonomy of a compact, connected hyperbolic surface $\Sigma$ with nonempty geodesic boundary.
Then there exist $t_0>0$ and a continuous family of representations $(j_t)_{0\leq t \leq t_0}$ with the following properties:
\begin{itemize}
  \item[$(a)$] $\lambda_{j_t}(\gamma)=(1-t)\,\lambda_{j_0}(\gamma)$ for any $t\in [0,t_0]$ and any $\gamma\in\Gamma$ corresponding to a boundary component of~$\Sigma$;
  \item[$(b)$] $\sup_{\gamma\in\Gamma\smallsetminus \{1\}}\, \frac{\lambda_{j_t}(\gamma)}{\lambda_{j_0}(\gamma)}<1$ for any $t\in (0,t_0]$;
  \item[$(c)$] $j_t(\gamma)=j_0(\gamma)+O(t)$ for any $\gamma\in\Gamma$ as $t\rightarrow 0$, where both sides are seen as $2\times 2$ real matrices with determinant~$1$;
  \item[$(d)$] for any compact subset $K$ of~$\HH^2$ projecting to the interior of the convex core of $j_0(\Gamma)\backslash\HH^2$, there exists $L>0$ such that
  $$d(p,f_t(p))\leq Lt$$
  for any $p\in K$, any $t\in[0,t_0]$, and any $1$-Lipschitz, $(j_0,j_t)$-equivariant map $f_t : \HH^2\rightarrow\HH^2$.
\end{itemize}
\end{lemma}

As in the proof of Lemma~\ref{lem:surfwithbound}, we construct the representations $j_t$ as holonomies of hyperbolic surfaces obtained from $j_0(\Gamma)\backslash\HH^2$ by deformation.
Now the deformation needs to be shortening instead of lengthening: we use \emph{negative} strip deformations.

\begin{proof}[Proof of Lemma~\ref{lem:surfwithbound2}]
We see $\Sigma$ as the convex core of $j_0(\Gamma)\backslash\HH^2$.
To shorten one boundary component $\beta$ of~$\Sigma$, choose a finite collection of disjoint, biinfinite geodesic arcs $\alpha_1,\dots, \alpha_n \subset j_0(\Gamma)\backslash\HH^2$, each crossing $\beta$ orthogonally twice, and subdividing $\Sigma$ into right-angled hexagons and one-holed right-angled bigons.
Near each $\alpha_i$, choose a second geodesic arc $\alpha'_i$, also crossing $\beta$ twice, such that $\alpha_i, \alpha'_i$ approach each other closest at some points $p_i, p'_i \in \Sigma$.
We take all arcs to be pairwise disjoint.
For every~$i$, delete the hyperbolic strip $A_i$ bounded by $\alpha_i$ and~$\alpha'_i$ and glue the arcs back together isometrically, identifying $p_i$ with~$p'_i$.
This yields a new complete hyperbolic surface, with a compact convex core, equipped with a natural $1$-Lipschitz map $\varsigma_t^{\beta}$ \emph{from} $j_0(\Gamma)\backslash\HH^2$, obtained by collapsing the strips $A_i$ to lines.
The set $\varsigma_t^{\beta}(\Sigma)$ is strictly contained in the new convex core.
The geodesic corresponding to~$\beta$ is shorter in the new surface than in~$\Sigma$; by adjusting the widths of the strips~$A_i$, we may assume that the ratio of lengths is $\frac{1}{1-t}$.
Note that the appropriate widths for this ratio are in $O(t)$ as $t\rightarrow 0$.
All lengths of geodesics corresponding to boundary components other than~$\beta$ are unchanged.

Repeat the construction, iteratively, for all boundary components $\beta_1,\dots, \beta_r$ of~$\Sigma$, in some arbitrary order.
We thus obtain a new complete hyperbolic surface $j_t(\Gamma)\backslash\HH^2$, with a compact convex core~$\Sigma_t$, such that $j_t$ satisfies~(a).
As in the proof of Lemma~\ref{lem:surfwithbound}, up to replacing each $j_t$ with a conjugate under $\PSL$, we may assume that (c) is satisfied.
To see that (b) and~(d) also hold, we use the $1$-Lipschitz map $\varsigma_t:=\varsigma^{\beta_r}_t\circ\dots\circ\varsigma^{\beta_1}_t$ from $\Sigma$ to~$\Sigma_t$ and argue as in the proof of Lemma~\ref{lem:surfwithbound}, switching $j_t$ and~$j_0$.
\end{proof}

As in Section~\ref{subsec:hallali1}, we write $\Sigma_g=\Sigma^{1}\cup\dots\cup\Sigma^{m}$, where $\Sigma^{i}$, for any $1\leq i\leq m$, is a compact surface with boundary that is
\begin{itemize}
  \item either a pair of pants labeled~$0$,
  \item or a full connected component of the subsurface of $\Sigma_g$ made of pants labeled $1$,
  \item or a full connected component of the subsurface of $\Sigma_g$ made of pants labeled $-1$.
\end{itemize}
Choose a small $\delta>0$ such that in all hyperbolic metrics on~$\Sigma_g$ which are close enough to that defined by~$j_0$, any simple geodesic entering the $\delta$-neighborhood of the geodesic representative of a cuff of our chosen pants decomposition crosses the cuff.
We use again the notation $\mathcal{C}_0,N_0^{\delta},K$ from Section~\ref{subsec:hallali1}.
Applying Lemma~\ref{lem:surfwithbound2} to $\Gamma^{i}:=\pi_1(\Sigma^{i})$ and $j_0^{i}:=j_0|_{\Gamma^{i}}$, we obtain continuous families of representations $(j_t^{i})_{0\leq t \leq t_0}$ for $1\leq i\leq m$ satisfying (a),(b),(c),(d), with a uniform constant $L>0$ for the compact set $K\subset\HH^2\smallsetminus\mathcal{C}_0$.
For any $t\geq 0$, using~(a), we can glue together the convex cores of the $j_t^{i}(\Gamma^{i})\backslash\HH^2$ following the same combinatorics as the~$\Sigma^{i}$: this gives a closed hyperbolic surface of genus~$g$, hence a holonomy representation $j_t\in\Hom(\Gamma_g,\PSL)$.
By (c), up to adjusting the twist parameters, we may assume that $j_t(\gamma)=j_0(\gamma)+O(t)$ for any $\gamma\in\Gamma_g$ as $t\rightarrow 0$.
Recall the notation $\mathcal{C}_t,N_t^{\delta}$ from Section~\ref{subsec:hallali1}.
By Proposition~\ref{prop:unif-prop-signs}, there exist a family $(\rho_t)_{0\leq t\leq t_0}\subset\Hom(\Gamma_g,\PSL)$ of non-Fuchsian representations and, for any $t\in [0,t_0]$, a $1$-Lip\-schitz, $(j_t,\rho_t)$-equivariant map $\varphi_t : \HH^2\rightarrow\HH^2$ that is an orientation-preserving (\resp orientation-reversing) isometry in restriction to any connected subset of~$\HH^2$ projecting to a union of pants labeled $1$ (\resp $-1$) in $j_t(\Gamma_g)\backslash\HH^2\simeq\Sigma_g$, such that
\begin{equation}\label{eqn:unif-bound-phit}
\Lip_p(\varphi_t) \leq C^{\ast} < 1
\end{equation}
for all $t\in [0,t_0]$ and all $p\in N_t^{\delta}$ projecting to a pair of pants labeled~$0$ in $j_t(\Gamma_g)\backslash\HH^2\simeq\Sigma_g$, for some $C^{\ast}<1$ independent of $p$ and~$t$.
We claim that for $t>0$ small enough,
\begin{equation}\label{eqn:lambda-rho_t-j}
\sup_{\gamma\in (\Gamma_g)_s}\ \frac{\lambda_{\rho_t}(\gamma)}{\lambda_{j_0}(\gamma)} < 1 ,
\end{equation}
which by Theorem~\ref{thm:contcont} is enough to prove that $[\rho_t]$ is strictly dominated by~$[j_0]$.
Indeed, by (b) and Theorem~\ref{thm:contcont}, for any $1\leq i\leq m$ and $t\in (0,t_0]$, there exists a $(j_t|_{\Gamma^{i}},j_0|_{\Gamma^{i}})$-equivariant map $f_t^{i} : \HH^2\rightarrow\HH^2$ with $\Lip(f_t^{i})<1$.
Let $f_t : (N_0^{\delta}\cup\mathcal{C}_0)\rightarrow\HH^2$ be a $(j_0,j_t)$-equivariant map such that
\begin{itemize} 
  \item $f_t=f_t^{i}$ on the component of $K$ projecting to~$\Sigma^{i}$, for all $1\leq i\leq m$;
  \item $f_t$ takes any geodesic line in~$\mathcal{C}_0$ to the corresponding line in~$\mathcal{C}_t$, multiplying all distances by the uniform factor $(1-t)$, and $d(x,f_t(x))\leq L_1t$ for all $x\in\mathcal{C}_0\cap K$, for some $L_1\geq 0$ independent of $x$ and~$t$.
\end{itemize}
Consider the $(j_0, \rho_t)$-equivariant map
$$G_t:=\varphi_t\circ f_t :\ (N_0^{\delta} \cup \mathcal{C}_0) \longrightarrow \HH^2 .$$
Any geodesic segment $I=[p,q]$ of~$\HH^2$ projecting to a closed geodesic of $j_0(\Gamma_g)\backslash\HH^2\simeq\Sigma_g$ may be decomposed into subsegments $I_1,\dots,I_n$ contained in~$N_0^{\delta}$ alternating with subsegments $I'_1,\dots,I'_n$ crossing connected components of $\HH^2\smallsetminus N_0^{\delta}$.
By contractivity of~$f_t$, the map $G_t$ has Lipschitz constant $<1$ on each connected component of~$N_0^{\delta}$, hence moves the endpoints of each~$I_k$ closer together by a uniform factor (independent of~$I$).
The subsegments $I'_k$ are treated by the following lemma, which implies \eqref{eqn:lambda-rho_t-j} and therefore completes the proof of the second statement of Theorem~\ref{thm:domin}.

\begin{lemma}[Analogue of Lemma~\ref{lem:hallali}]\label{lem:hallali2}
For small enough $t>0$, there exists $C<1$ such that for all $p,q\in\partial N_0^{\delta}$ lying at distance $\delta$ from a line $\ell_0\subset\mathcal{C}_0$, on opposite sides of~$\ell_0$,
$$d(G_t(p),G_t(q)) \leq C\, d(p,q).$$
\end{lemma}

The proof of Lemma~\ref{lem:hallali2} uses the following observation, which is identical to Observation~\ref{obs:additive} after exchanging $j_0$ and~$j_t$.

\begin{observation}\label{obs:additive2}
There exists $L'\geq 0$ such that for any small enough $t\geq 0$, any $p\in\partial N_0^{\delta}$ at distance $\delta$ from a geodesic $\ell_0\subset\mathcal{C}_0$, and any $x\in\ell_0$,
\begin{equation}\label{eqn:additive2}
d(f_t(p),f_t(x)) \leq (1-t) \, d(p,x) + L' t .
\end{equation}
\end{observation}

\begin{proof}[Proof of Lemma~\ref{lem:hallali2}]
We argue as in the proof of Lemma~\ref{lem:hallali}, but switch $j_0$ and $j_t$ and use \eqref{eqn:unif-bound-phit} to obtain the analogue 
\begin{equation*}
d\big(\varphi_t \circ f_t(p), \varphi_t \circ f_t(x)\big) \leq d(f_t(p), f_t(x)) - \varepsilon
\end{equation*}
of \eqref{eqn:rabiot} when $[p,x]$ projects to a pair of pants labeled~$0$ in $j_0(\Gamma_g)\backslash\HH^2\simeq\nolinebreak\Sigma_g$.
\end{proof}

\vspace{0.5cm}

\end{document}